\documentclass[12pt, twoside]{article}
\usepackage{amsmath,amsthm,amssymb}
\usepackage{times}
\usepackage{enumerate}
\usepackage[utf8x]{inputenc}
\usepackage{mathtools}
\usepackage{xcolor}
\usepackage{bbm}
\usepackage{verbatim}
\usepackage{hyperref}
\usepackage{tikz-cd} 
\usepackage{graphicx}
\usepackage{cleveref}

\usepackage[ english ]{ babel }

\usepackage[colorinlistoftodos]{todonotes} 

\theoremstyle{definition}
\newtheorem{thm}{Theorem}[section]
\newtheorem{cor}[thm]{Corollary}

\newtheorem{definition}[thm]{Definition}
\newtheorem{remark}[thm]{Remark}

\newtheorem{ass}[thm]{Assumption}
\newtheorem{problem}{Problem}[section]
\newtheorem{lemma}{Lemma}[section]

\newtheorem{theorem}{Theorem}[section]

\numberwithin{equation}{section}

\newcommand{\subjclass}[1]{\bigskip\noindent\emph{2010 Mathematics Subject Classification:}\enspace#1}
\newcommand{\keywords}[1]{\noindent\emph{Keywords:}\enspace#1}

\newcommand{\R}{\mathbb{R}}

\newcommand\restr[2]{{
  \left.\kern-\nulldelimiterspace
  #1
  \vphantom{\big|}
  \right|_{#2}  }}

\newcommand{\Pp}{\mathbb{P}}

\newcommand{\lz}{\left( }
\newcommand{\dz}{\right)}
\newcommand{\Vel}{\textbf{v}}

\definecolor{airforceblue}{rgb}{0.36, 0.54, 0.66}
\definecolor{ballblue}{rgb}{0.13, 0.67, 0.8}

\title{ Linear Parabolic Problems in Random Moving Domains}

\author{Ana Djurdjevac\\
Institut für Mathematik\\
Techinsche  Universität Berlin\\
adjurdjevac@math.tu-berlin.de\\
}

\date{}

\begin{document}

\maketitle
\begin{abstract}
We consider linear parabolic equations on a random non-cylindrical domain. Utilizing the domain mapping method, we write the problem as a partial differential equation with random coefficients on a cylindrical deterministic domain. Exploiting the deterministic results concerning equations on  non-cylindrical domains, we state the necessary assumptions about the velocity filed and in addition, about the flow transformation that this field generates. In this paper we consider both cases, the uniformly bounded with respect to the sample and log-normal type transformation. In addition, we give an explicit example of a log-normal type transformation and prove that it does not satisfy the uniformly bounded condition. We define a general framework for considering linear parabolic problems on random non-cylindrical domains. As the first example, we consider the heat equation on a random tube domain and prove its well-posedness. Moreover, as the other example we consider the parabolic Stokes equation which illustrates the case when it is not enough just to study the plain-back transformation of the function, but instead to consider for example the Piola type transformation, in order to keep the divergence free property. 
\end{abstract}



\section{Introduction}
Partial differential equations (PDEs) appear in the mathematical modeling of a great variety
of processes. Often in these models  uncertainty appears for various reasons, such as  incomplete knowledge about the given data. The given data can be a source term, an initial state, parameters, a domain etc. In this work we study the situation where the uncertainty of the model comes from the geometrical aspect. For example, the computational domain is often given by scanning, or some other digital imaging technique with limited resolution which leads to the variance between the shape of the real body and the model (for a mathematical model of this problem see \cite{BC}). A well established and efficient way to deal with this problem is to adopt the probabilistic approach, i.e. construct models of geometrical uncertainty and describe the phenomena by PDEs on a random domain.  More precisely, one considers the fixed initial deterministic domain $D_0 \subset \R^d$ and its evolution in a time interval $[0,\tau]$ by a random velocity $V$, defined on a probability space $(\Omega, \mathcal{A}, \mathbb{P})$. In this way one obtains a non-cylindrical, i.e. time-dependent, random domain 
$$Q(\omega) : = \bigcup_{t \in (0,\tau)} D_t(\omega)\times \{ t \} ,$$
also known as a  random tube domain, where $\omega$ is a sample from $\Omega$. 

Random domains appear in many applications, such as biology, surface imaging, manufacturing of nano-devices etc. One particular application example occurs in the wind engineering as presented in \cite{CF}. More precisely, the authors study how small uncertain geometric changes in the Sunshine Skyway Bridge deck affect its aerodynamic behavior. The geometric uncertainty of the bridge is due to its specific construction and wind effect.   This model results in a parabolic PDE on a random domain. Another class of examples comes from the modeling in Cardiovascular Biomechanics. In particular there are papers that consider uncertainty in these models, that are mainly presented by flow equations \cite{SMapp1,SMapp2} and of specific interest is  the uncertainty that comes from the geometry.  Another application was suggested in \cite{TX2}, where authors suggest  to treat rough surfaces, i.e. highly irregular surfaces such as glaciers, as random domains.  In the light of their practical relevance, the analysis and numerical analysis of random domains have been considered by many authors, see \cite{CC,CCNT,HPS,HSSS,XT}. Notice that of the application point of view, in particular  biological applications, it is of big interest to consider these equations on curved domains. Elliptic PDEs on random surface domains have been considered in \cite{CDJE}.

In Figure \ref{figure:tube}, we visualize the difference between the deterministic cylindrical domain, the random cylindrical domain and the random non-cylindrical domain. The first plot presents a standard cylindrical domain. The second one is a realization of a random tube given by 
\[
\mathbb{S}^1 \ni (x_0, y_0) \mapsto (x(\omega),y(\omega)) := (2x_0 Y_1(\omega), 3 y_0 Y_2(\omega)) \in D(\omega)
\]
where $Y_1, Y_2 \sim \mathcal{U}(0,1)$ are independent RVs. The last two plots are two realizations of a random non-cylindrical tube defined by
\begin{eqnarray*}
&\mathbb{S}^1 \ni (x_0, y_0) \mapsto (x(\omega,t),y(\omega,t)) := \\
&\!\!\!\!\!\!(Y_1(\omega) (\sin(Y_2(\omega))\! + \!1.5 ) x_0\! +\! 0.3 \cos (Y_3(\omega)t), Y_4(\omega) (\sin(Y_5(\omega)) \!+ \!1.5 ) y_0 \!+ \!0.3 \sin (Y_6(\omega)t))\!\! \in  \!\!D_t(\omega)
\end{eqnarray*}
where $Y_1,\dots, Y_6 \sim \mathcal{U}(0,1)$ are independent RVs.

\begin{figure}[htb]
  \centering
  \includegraphics[width=0.6\textwidth]{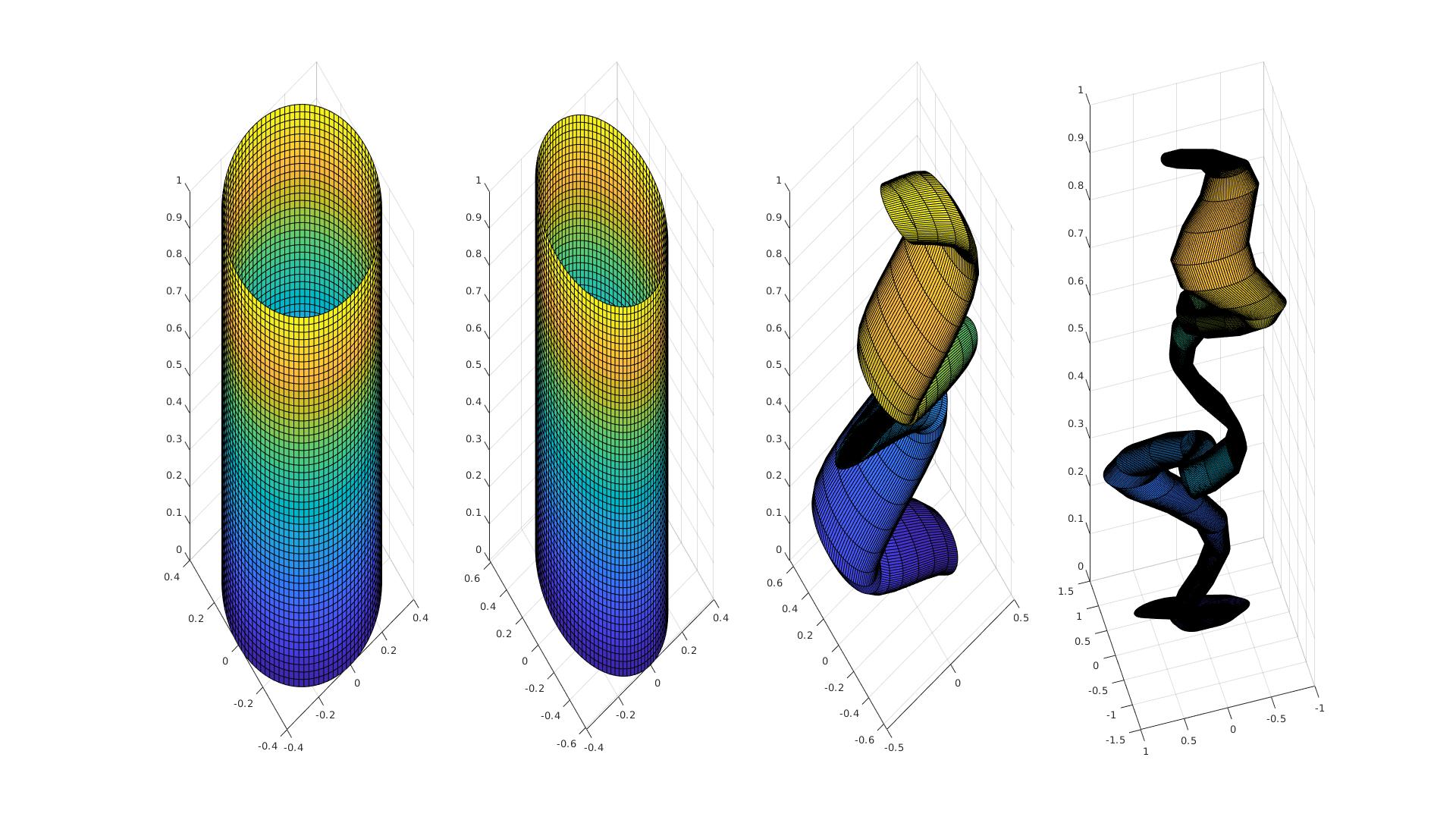}
  \caption[Cylindrical domain, realizations of a random cylindrical domain and  of a random non-cylindrical domain]{Cylindrical domain, realization of a random cylindrical domain and realizations of random non-cylindrical domains, respectively.} \label{figure:tube}
\end{figure}

The approach that we consider in this paper is known as the domain mapping method \cite{HPS,XT}. The other well-known approaches are the perturbation method  (cf. \cite{HSS}), eXtended stochastic FEM  \cite{NCSM} and fictitious domain approach \cite{CK} . The domain mapping method   requires knowledge about the transformation field $T$ on the closer of the fixed initial  domain $D_0$:
\[
T(\omega) : \overline{ D_0} \mapsto \R^d.
\]
The main idea of this method is to reformulate the PDE on the random domain into the PDE with random coefficients on a fixed reference domain. This reformulation allows us to apply numerous available analysis and numerical methods for solving random PDEs and to avoid the construction of a new mesh for every realization of a random  domain.

We assume that we are given a deterministic domain $D_0$ that evolves with a given random velocity field $V$, and as a result we build a random tube. To a random velocity, we will associate its flow $T_V$ that will map a domain $D_0$ into a random domain $D_t(\omega)$ at a time $t$.

Notice that, in the previous work on random domains, mainly elliptic PDEs have been considered. Very few papers consider parabolic PDEs on random domains, such as \cite{CC,CF, TX2}. In addition, to the best of our knowledge, there are no results on random domains that change in time, which is exactly the goal of this work. This paper is based on the draft presented in \cite{Dj}, that resulted from the thesis \cite{DjPhD}. We will consider the well-posedness of linear parabolic PDEs posed on random moving domains and necessary conditions about the initial data, the random velocity field and the induced transformation, that will ensure the  well-posedness of the equation. 

In the deterministic case, PDEs on the so-called non-cylindrical domains, i.e. domains changing in time, are a well-established topic regarded  analysis and numerical analysis  (see \cite{BG, BHT19,CDPZ,CRB,DPZ,KRS,LMZ}),  with numerous applications. Various physical examples concerning  phenomena on time dependent spatial domains are presented in the survey article \cite{KK}. 
Some of the examples are: fluid flows with a free or moving solid-fluid interface, the Friedmann model in astrophysics that describes the scaling of key densities in the Universe with its time-dependent expansion factor,  and many examples of biological processes that involve time-dependent domains, such as the formation of patterns and shapes in biology. 
In \cite{KRS,LMZ}  authors focus on appropriate formulation of the heat equation on a random domain and on proving the existence and uniqueness of strong, weak and ultra weak solutions, as well as providing  energy estimates.  These papers  use coordinate transformation to reformulate the PDE into a cylindrical domain and   Lions' general  theory for proving the well-posedness. Moreover, in \cite{DPZ,DzTube,DZEul} similar results were obtained but with a greater focus on the connection of the non-cylindrical domain and the velocity field. Since we are particularly interested in how the velocity field induces a non-cylindrical domain, we will  follow this approach. In this paper we exploit these deterministic results, to derive necessary regularity results in the random setting about the given velocity field. In addition to it, we consider assumptions  about the measurability  and $L^p(\Omega)$ bounds. The random velocity field could be presented in many ways, for example in numerics it is common to assume that the velocity field is given by its truncated Karhunen-Loeve expansion.

In this paper, we  consider both cases, the uniformly bounded in $\omega$ and log-normal type, of the flow transformation. To the best of our knowledge up to now all the analysis in the random domain setting, has been regarding the uniformly bounded domain transformation. We construct an explicit example of the transformation that does not satisfy this condition. The type transformation is of the  exponential of the regularized Brownian motion. We prove that this type of transformation still satisfies necessary regularity assumptions, but not uniformly bounded assumption. In \cite{LognormalAppl} the authors present the big variety of applications of log-normal field in ecology, in particular in growth models. One could then consider PDEs whose computational domain is given by these growth models.
Exploiting domain mapping method we prove well-posedness in both types of previously mentioned random transformations of the domain. The first step is to transform the equation to the deterministic cylindrical domain, using the pull back transformation that is a transformation, of the so-called plain pull-back of the considered function $u$ on the random domain $Q_T(\omega)$, i.e. 
$$\hat u(\omega, x, t) := \mathcal{F}u(\omega, t, T_t(x)) \quad \omega \in \Omega, t \in [0,\tau], x \in D_0,$$
where $\mathcal{F}$ is a suitable transformation.
As a standard example of a linear parabolic PDE we consider the heat equation on the tube domain, and in this case $\mathcal{F}$ is identity.
 In the uniformly bounded setting the well-posedness, is a direct consequence of the general theory of parabolic PDEs. In the log-normal type case, one needs in addition to prove measurability of the solution, and its $L^p(\Omega)$ bounds. To achieve that, we exploit the proof of the general deterministic result in order to exactly determine the constants that appear in a path-wise setting. Under suitable assumptions, we can control these constants and prove well-posedness of the pulled-back equation.

As a next example, we consider the parabolic Stokes equation. Parabolic Stokes and Navier-Stokes equations on a non-cylindrical domain have been investigated by many authors. In particular they studied the well-posedness and regularity results.  This equation on a non-cylindrical deterministic domain has been considered by many authors \cite{WI, Saal}.  This example illustrates that it is not always enough to choose a plain pull-back transformation, as in the case of the heat equation, but instead one has to cleverly choose the transformation $\mathcal{F}$ in order to preserve certain properties (such as divergence free) or to obtain more simple form of the pulled-back equation. In this case, we choose a Piola type transformation and derive the pulled-back equation. The well-posedness of this equation is briefly commented, and its proof is the subject of future work.  

This paper is organized as follows. In  Section 2, we introduce the general function spaces that will be exploited in the formulation of PDEs on random moving domains. Namely, on the cylindrical domain we consider the standard Sobolev-Bochner type spaces, and for the path-wise consideration on the random tube we utilize the general setting for the PDEs on moving domains defined in \cite{EAS}. In Section 3 we define precisely  what  we mean by a random tube, in particular we specify the needed regularity assumptions based on the work presented in \cite{DZEul,DzTube}.  In addition, we give an example of a transformation field that does not satisfy the uniformly bounded assumption. Section 4 is devoted to   the heat equation on a random tube, as a typical example of a linear parabolic equation. We prove its well-posedness, in the case of uniformly-bounded and long-normal type transformation, exploiting the  domain decomposition method. In Section 5 we consider the parabolic Stokes equation, where more complicated Piola type transformation is applied. We conclude in Section 6, by  giving a brief conclusion of the paper and suggest further possible directions of a research.

\section{General Setting}

Let $(\Omega, \mathcal{F}, \mathbb{P})$ be a complete probability space with a sample space $\Omega$, a $\sigma$-algebra of
events $\mathcal{F}$  and a probability $ \mathbb{P}$. In addition, we assume that $L^2(\Omega)$ is a
separable space. For this assumption it suffices to assume that $(\Omega, \mathcal{F}, \Pp)$ is a separable measure space  \cite[Theorem 4.13]{sepBrezis}, i.e. that $\mathcal{F}$ is generated by a countable collection of subsets. Under this assumption, for any separable Hilbert space $H$, it holds (see \cite[Theorem II.10]{RS})
\begin{equation}\label{tensor}
L^2(\Omega) \otimes H \cong L^2(\Omega, H).
\end{equation}

We only consider a fixed finite
time interval $[0,\tau]$, where $\tau \in (0,\infty).$ Furthermore, we  denote by $D([0,\tau],\mathcal{H})$ the space of all $C^\infty$-smooth $\mathcal{H}$-valued  functions with compact support in $[0,\tau].$ In the special case, when $\mathcal{H}$ 
equals $\R$, we just write  
 $D([0,\tau])$.
We will reuse the same constants $C$ in calculations multiple times if their exact value is not important. When there is no confusion,  integrals will be usually written without measure. Let $D_0 \subset \R^d$ be an open, bounded domain with a Lipschitz boundary. The curly notation for spaces, such as $\mathcal{V}, \mathcal{H}$ and hat for functions $\hat{u}$, will be used for the deterministic cylindrical domain. When we utilize  some deterministic result  in a path-wise sense, if there is no confusion, we omit writing the dependence on the sample $\omega$ to simplify the notation.

\subsection{Function spaces on the cylindrical domain} \label{sec:Bochner_sp}

The Bochner spaces are straightforward generalization of the Lebesgue spaces to  Banach space valued functions and are natural spaces to use for the pulled-back formulation of a PDE on a cylindrical (deterministic) domain.
One can define the integral of an $E$-valued random variable $X: \Omega \to E $, where $(E, \mathcal{B}(E))$ is a separable Banach space.
The precise definition of Bochner spaces and its properties can be found for example in \cite{PZ}, and here we follow the summary of results from \cite{EAS}.
Let $$\mathcal{V} \xhookrightarrow[i]{} \mathcal{H} \cong \mathcal{H}^*\xhookrightarrow[i']{} \mathcal{V}^*$$ be a Gelfand triple. Every vector-valued distribution $u \in   L^2(0,\tau; \mathcal{H})$ defines a vector-valued distribution $\mathrm{S}_u: D((0,\tau)) \to\mathcal{H}$ through the $\mathcal{H}$-valued integral 
\[
\varphi \mapsto \int_0^\tau y(t) \varphi(t) dt. 
\]
We will identify $\mathrm{S}_u$ and $u$. 
We say  that $u \in L^2(0,\tau; \mathcal{H})$ has a {\textit{weak derivative}} $u' \in L^2(0,\tau; \mathcal{H}^*) $
if there exists $w \in L^2(0,\tau; \mathcal{H}^*)$ such that
\begin{equation}\label{weakDer}
\textrm{S}'_u(\xi) = \int_0^\tau \xi '(t) (u(t),v)_{\mathcal{L}} = - \int_0^\tau \xi(t) \left< w(t), v \right>_{\mathcal{H}^*,\mathcal{H}}, \quad \forall \xi \in D(0,\tau),\, \quad \forall v \in \mathcal{H}
\end{equation}
and we write $w=u'$.
Recall that the standard Sobolev-Bochner space is defined as
\begin{equation}\label{def:stdSBsp}
\mathcal{W}(\mathcal{V}, \mathcal{V}^*) := \{ u \in L^2(0,\tau; \mathcal{V}) \mid u' \in L^2(0,\tau;\mathcal{V}^*) \}.
\end{equation}
The space $\mathcal{W}(\mathcal{H}, \mathcal{H}^*)$ is a Hilbert space with the inner product defined via:
\begin{equation*} 
\lz u,v\dz_{\mathcal{W}(\mathcal{V},\mathcal{V}^*)} := \int_0^\tau  \lz u(t), v(t)\dz_{\mathcal{V}} + \int_0^\tau (u'(t), v'(t))_{\mathcal{V^*}}.
\end{equation*}
\begin{theorem}\label{thm:prop_std_SB_sp} The following  properties of the space $\mathcal{W}(\mathcal{H}, \mathcal{H}^*)$ hold
\begin{enumerate}
\item[i)] The embedding $\mathcal{W}(\mathcal{V}, \mathcal{V}^*) \subset C([0,\tau],\mathcal{H})$ is continuous.

\item[ii)] 
The embedding $ \mathcal{W}(\mathcal{V}, \mathcal{V}^*) \subset D([0,\tau],\mathcal{V})$  is dense.

\item[iii)] Let $u,v \in \mathcal{W}(\mathcal{V}, \mathcal{V}^*)$, then the mapping
$t \mapsto (u(t),v(t))_{\mathcal{H}}$
is absolutely continuous on $[0,\tau]$ and 
\[
\frac{d}{dt} (u(t),v(t))_{\mathcal{H}} = \left<u'(t),v(t) \right>_{\mathcal{V}^*,\mathcal{V}} + \left<u(t),v'(t) \right>_{\mathcal{V},\mathcal{V}^*}
\]
holds for almost every $t \in [0,\tau]$. The last expression implies the integration by parts formula
\[
(u(\tau),v(\tau))_{\mathcal{H}} - (u(0),v(0))_{\mathcal{H}} = \int_0^\tau \left<u'(t),v(t) \right>_{\mathcal{V	}^*,\mathcal{V}} + \int_0^\tau  \left<u(t),v'(t) \right>_{\mathcal{V},\mathcal{V}^*}.
\]
\end{enumerate}
\end{theorem}
\begin{proof}
For the density result see \cite[Theorem 2.1]{LM} and for other statements see \cite{Show}.
\end{proof}

In the definition of a weak material derivative, we  utilize that the weak derivative can be characterized in terms of vector-valued test-functions. We state this standard result for completeness. 
\begin{theorem}
The weak derivative condition (\ref{weakDer}) is equivalent to 
\begin{equation}
\int_0^\tau (u(t), \psi'(t))_{\mathcal{H}} = - \int_0^\tau \left< u'(t), \psi(t) \right>_{\mathcal{V}^*,\mathcal{V}} \quad \forall \psi \in D((0,\tau),\mathcal{V}). \label{ch_wd}
\end{equation}
\end{theorem}
\begin{proof}
The direct implication follows from Theorem \ref{thm:prop_std_SB_sp}, iii). To see that (\ref{ch_wd}) implies (\ref{weakDer}), test (\ref{ch_wd}) with $\xi v \in D((0,\tau),\mathcal{H})$, where $\xi \in D((0,\tau))$ and $v \in \mathcal{H}$.
\end{proof}

\subsection{Bochner-type spaces}

In order to treat the evolving family of Hilbert spaces $X= (X(t))_{t\in[0,\tau]}$, the idea is to connect the space $X(t)$ at any time $t \in [0,\tau]$ with some fixed space, for example with the initial space $X(0)$. Thus we construct the family of functions $\phi_t : X(0) \rightarrow X(t)$, which we call \textit{the (plain) pushforward map}. We denote the inverse of $\phi_t$ by $\phi_{-t} : X(t) \rightarrow X(0)$ and call it \textit{the (plain) pullback map}.  We want to introduce the special Bochner-type function space such that for every $t \in [0,\tau]$ we have $u(t) \in X(t)$. These spaces are introduced in \cite{EAS} and we exploit them   to define the solution space for the path-wise problem on a random tube domain. Note that in a random setting we have for every path $\omega$ a random family  $X(\omega)= (X(\omega, t))_{t\in[0,\tau]}$, where $X(\omega, t)$ denotes an appropriate function space on a random domain $D_t(\omega)$ and $X(0)$ is a fixed initial deterministic space. We state the main definitions and results from \cite{EAS}, that we apply in our setting.

\begin{definition}\label{def:comp_spaces}
The pair $\{X, (\phi_t)_{t \in [0,\tau]} \}$ is \emph{compatible} if the following conditions hold:

$\bullet$ for every $t \in [0,\tau], \,\,  \phi_t$ is a linear homeomorphism such that $\phi_0$ is the identity map

$\bullet$ there exists a constant $C_X$ which is independent of $t$ such that
\begin{eqnarray*}
\| \phi_tu \|_{X(t)} &\leq& C_X \| u \|_{X(0)} \quad \text{ for every } u \in X(0)  \\
\| \phi_{-t}u \|_{X(0)} & \leq& C_X \| u \|_{X(t)} \quad \text{ for every } u \in X(t)
\end{eqnarray*}

$\bullet$ the map $t \mapsto \| \phi_tu \|_{X(t)}$ is continuous for every $u \in X(0)$.
\end{definition}

Note that for the given family $(X(t))_{t \in (0,\tau)}$ there are usually many different mappings $\phi_t$ such that the pair $\{X, (\phi_t)_{t \in [0,\tau]} \}$ is compatible.

We  denote the dual operator of $\phi_t$ by $\phi_t^* : X^*(t) \rightarrow X^*(0)$. As a consequence of the previous conditions, we obtain that $\phi_t^*$ and
its inverse are also linear homeomorphisms which satisfy the following conditions
\begin{eqnarray*}
\| \phi_t^* f\|_{X^*(0)} &\leq& C_X \| f \|_{X^*(t)} \quad \text{for every } f \in X^*(t) \\
\| \phi_{-t}^*f \|_{X^*(t)} &\leq& C_X \| f \|_{X^*(0)} \quad \text{for every } f \in X^*(0).
\end{eqnarray*}

Now we can define the spaces of time-dependent functions whose domain is also changing in time by requiring that the pull-back of $u$ belongs to the fixed initial space.

\begin{definition}\label{def:L_X}
For a compatible pair $\{X, (\phi_t)_{t \in [0,\tau]}\}$ we define spaces
\begin{align*}
L^2_X &:= \left \{ u: [0,\tau] \ni t \mapsto (\bar{u}(t),t) \in  \!\!\! \underset{s\in[0,\tau]}{\bigcup}\!
X(s) \times \{s\}
\mid \phi_{- (\cdot) } \bar{u}(\cdot ) \in L^2(0,\tau;X(0)) \right \},  \\
L^2_{X^*} &:= \left \{ f: [0,\tau]  \ni t  \mapsto  (\bar{f}(t),t) \in  \!\!\! \underset{s\in[0,\tau]}{\bigcup}\!X^*(s) \times \{s\}
\mid 
\phi_{- (\cdot)}\bar{f}(\cdot) \in L^2(0,\tau;X^*(0)) \right \} .
\end{align*}
\end{definition}
Like the standard Bochner spaces, these spaces
consist of equivalence classes of functions agreeing almost everywhere in $[0, \tau ].$ Observe that previous spaces strongly depend on the  map $\phi_t$.
In the following we  identify $u(t)=(\overline{u}(t),t)$ with $\overline{u}(t)$, for brevity of notation.

In order to understand these spaces better, we  state their most important properties, stated in \cite[Lemma 2.10, Lemma 2.11]{EAS}. 

\begin{lemma}\label{isom_with_Boh_sp}(\!The isomorphism with standard Bochner spaces and the equivalence of norms) The maps 
\begin{align*}
L^2(0,\tau; X_0) \ni u &\mapsto \phi_{(\cdot)} u(\cdot) \in L^2_X \\
L^2(0,\tau; X_0^*) \ni f &\mapsto \phi_{-(\cdot)} f(\cdot) \in L^2_{X^*} 
\end{align*}
are isomorphisms.
Furthermore, the equivalence of norms holds
\begin{align*}
\frac{1}{C_X} \| u \|_{L^2_X} &\leq \|  \phi_{(\cdot)} u(\cdot) \|_{L^2(0,\tau; X_0)} \leq C_X \| u \|_{L^2_X} \quad \forall u \in L^2_X \\
\frac{1}{C_X} \| f \|_{L^2_{X^*}} &\leq \|  \phi_{-(\cdot)} f(\cdot) \|_{L^2(0,\tau; X_0^*)} \leq C_X \| f \|_{L^2_X} \quad \forall f \in L^2_{X^*} .
\end{align*} 
\end{lemma}

The spaces $L^2_X$ and $L^2_{X^*}$ are separable Hilbert spaces (\cite[Corollary 2.11]{EAS}) with the inner product defined as
\begin{eqnarray*}
(u,v)_{L^2_X} := \int_{0}^{\tau}(u(t),v(t))_{X(t)} dt \\
(f,g)_{L^2_{X^*}} := \int_{0}^{\tau}(f(t),g(t))_{X^*(t)} dt.
\end{eqnarray*}

For $f \in L^2_{X^*}$ and $u \in L^2_X$ the map
\[
t \mapsto \left< f(t), u(t) \right>_{X^*(t), X(t)}
\]
is integrable on $[0,\tau]$, see \cite[Lemma 2.13]{EAS}. Utilizing the integrability of this map and Fubini's theorem,  in \cite[Lemma 2.15]{EAS} the authors prove that the spaces $L^2_{X^*}$ and  $(L^2_X)^*$ are isometrically isomorphic.
 Furthermore, the  duality pairing of $f \in L^2_{X^*}$  with $u \in L^2_X$ is given by
\[
\left< f, u \right>_{L^2_{X^*}, L^2_X} =\int_0^\tau \left< f, u \right>_{X^*(t), X(t)} dt.
\]

\subsection{Function spaces on a random non-cylindrical domains}\label{sec:solution_space}

In order to treat the path-wise formulation on a random non-cylindrical domain we will consider for every $t$ and $\omega$,  a path-wise Gelfand triple
\[
V(\omega, t) \xhookrightarrow[i]{} H(\omega, t) \cong {H}^*(\omega, t)\xhookrightarrow[i']{} V^*(\omega, t).
\]
Given random velocity filed $\Vel$ induces a random flow $T_\Vel$ such that   $T_\Vel( \omega,t): D_0 \to D_t(\omega)$, more details on this will be presented in the next section. For every $\omega \in \Omega$ we  define the path-wise pullback operator 
$\phi_{-t}(\omega):  H(\omega, t) \rightarrow H(0)$ in the following way
\begin{equation}\label{path_flow}
(\phi_{-t}(\omega)u(\omega))(x) := u(\omega)T_\Vel(\omega,t,x) \quad \text{for every } x \in D_0, \,\, \omega \in \Omega.
\end{equation}

\begin{lemma}\label{compatibility}
For every $\omega$, the pairs $\left(H(\omega), (\phi_t(\omega))_{t \in [0,\tau]} \right)$ and $\left(V(\omega), (\restr{\phi_t(\omega)}{V_0})_{t \in [0,\tau]}\right)$ are compatible.
\end{lemma}
\begin{proof}
The proof is similar to the proof of \cite[Lemma 3.3]{V} and the constants $C_X$ from the Definition \ref{def:comp_spaces} do not depend on the sample $\omega$, under suitable uniform bound assumption about the velocity field. However, note that in the log-normal case, constants do depend on the sample. 
\end{proof} 

From the previous lemma and Definition \ref{def:L_X} it follows that we can consider spaces $L^2_{V(\omega)}$ and $L^2_{H(\omega)}$. To define the solution space for parabolic PDEs on a random non-cylindrical domain, we introduce a material derivative which  takes into account the spatial movement of the domain. In order to  consider the material derivative of random functions, we apply the abstract setting from  \cite[Chapter 2.4]{EAS}.

\noindent
We define the spaces of pushed-forward continuously differentiable functions by
\begin{equation*}
\mathcal{C}_X^j := \{ u \in L^2_X  \mid \phi_{-(\cdot)}u(\cdot) \in \mathcal{C}^j([0,\tau], X_0\} \, \text{ for } j \in \{0,1,\dots \}.
\end{equation*}
In a random setting, the evolving family $X_t$  also depends on a sample on $\omega$, in a sense that $X_t = X_t(\omega)$. Since we  use the path-wise perspective for the equations on a random domain, 
for the simplicity of the notation we  formulate the next results in the deterministic framework, i.e. omit writing $\omega$.
\begin{definition}
For $u \in \mathcal{C}_H^1$ \emph{the strong material derivative} $\dot{u} \in  \mathcal{C}_X^0$ is defined by
\begin{equation}\label{material derivative}
\dot{u}(t) = \phi_t \left(\frac{d}{dt}\phi_{-t}u(t) \right) \quad \text{for every } t \in [0,\tau].
\end{equation}
\end{definition}
We can derive the following explicit formula for the {\textit{strong material derivative}}, (see \cite{EAS})
\begin{equation}
\dot{u}(t,\omega,x) = u_t(t,\omega,x) + \nabla u(t,\omega,x) \cdot \Vel(\omega,t,x), \quad \text{for every } x \in D(\omega,t), \, \omega \in \Omega, \label{strong md}
\end{equation}
where $\Vel$ is the velocity of the domain.

Just as in the static case, it might happen that the equation does not have a solution if requesting  $u \in  \mathcal{C}_H^1$. One can define a weak material derivative that needs less regularity. In addition to the case when we consider a fixed domain, an extra term, defined by the bilinear form $c$, that  takes into account the movement of the domain, appears. 

\begin{definition}
We say that $\partial^\bullet{u} \in L^2_{V^*}$ is \emph{a weak material derivative} of  $u \in L^2_V$ if and only if
\begin{align}\label{def:weak mat der}
\begin{split}
\int_0^\tau \left< \partial^\bullet{u}(t), \eta(t) \right>_{V^*(t),V(t)} = -\int_0^\tau \left( u(t), \dot{\eta}(t) \right)_{H(t)} - \int_0^\tau c(t;u(t),\eta(t)) 
\end{split}
\end{align}
holds for all $\eta \in \mathcal{D}_\Vel(0,\tau) = \{ \eta \in L^2_V \mid \phi_{-( \cdot )}\eta(\cdot) \in \mathcal{D}((0,\tau);V_0) \} $ and  $c(t, u(t),\eta(t)) = \int_{D(t)} u(t,x) \eta(t,x) \nabla_{\Gamma(t)} \cdot \Vel(t,x).$
\end{definition}
In a probabilistic setting the bilinear form $c(t, u(t),\eta(t)) $ is defined in the analogue way, including the dependence on a sample. Note that it can be directly shown that if it exists, the weak material derivative is unique and every strong material derivative is also a weak material derivative.

The solution space, based on the general framework, is defined by
\begin{equation}\label{def:path_sol_space}
W(V(\omega),V^*(\omega)) := \{ u \in L^2_{V(\omega)} \mid \partial^\bullet{u} \in L^2_{V^*(\omega)} \}.
\end{equation}
In order to prove that the solution space is a Hilbert space and  that it has some additional properties, one can  connect $W(V,V^*)$ with the standard Sobolev-Bochner space $\mathcal{W}(V_0, V_0^*)$ defined by (\ref{def:stdSBsp}) for which these properties are known. 
The previous two types of spaces are connected in a natural way, i.e. that the pull-back of the functions from the solution space belongs to the Sobolev-Bochner space and vice versa. In addition, we also have the equivalence of the norms. In this case we say that there exists an {\textit{evolving space equivalence}} between the spaces $ W(H,H^*)$ and $\mathcal{W}(H_0,H_0^*)$ (for the proof see \cite[Theorem 2.33]{EAS}). As the consequence of this equivalence and Theorem \ref{thm:prop_std_SB_sp} we have the following results.

\begin{lemma}The solution space $W(V,V^*)$ is a Hilbert space with the inner product defined via
\begin{equation*}
(u,v)_{W(V,V^*)} = \int_0^\tau  (u(t), v(t))_{V(t)} + \int_0^\tau  (\partial^\bullet u(t), \partial^\bullet v(t))_{V^*(t)}.
\end{equation*}
\end{lemma}

\begin{lemma}\label{lem:sol_space}  The following statements hold:
\begin{enumerate}
\item[i)] Space $W(V,V^*)$ is embedded into $C_H^0$. 

\item[ii)] The embedding $D_V([0,\tau]) \subset W(V,V^*)$ is dense.

\item[iii)] For every $u \in W(V,V^*)$ it holds 
$
\max_{t \in [0,\tau]} \| u(t) \|_{H(t)} \leq C \| u \|_{W(V,V^*)}.
$
\end{enumerate}
\end{lemma} 

 As a consequence, the evaluation $t \mapsto u(t)$ is well-defined and we are able to specify initial conditions for the PDE. Furthermore, we can define the subspace
 \begin{equation} \label{zeroSolSpace}
 W_0(V,V^*) := \{ u \in W(V,V^*) | u(0) = 0 \}.
 \end{equation}
Note that $W_0(V, V^*)$ is a Hilbert space, as a closed linear subspace of $W(V,V^*)$.

To discuss the case when more regularity is feasible, in particular if  the weak derivative of a function has more regularity, we define another function space.
\begin{definition}
Let
\begin{equation}\label{subspace}
W(V,H):= \{ u \in L^2_V \, | \, \partial^\bullet u \in L^2_H \}.
\end{equation}
\end{definition}

In order to prove the properties of the previous space, similarly as before,  we  connect $W(V,H)$ with the standard Sobolev-Bochner space $\mathcal{W}(V_0,H_0) \equiv \{ v \in L^2(0,\tau;V_0) \, | \, v' \! \in L^2(0,\tau;H_0)\}$. As a consequence, 
$W(V(\omega),H(\omega))$ is a Hilbert space for every $\omega \in \Omega$.

\section{Random tubes}\label{sec:randomTubes}

Let $D_0 \subset \R^d$ be an open, bounded domain with a Lipschitz boundary. Furthermore, let 
$\Vel: \Omega \times [0,\tau] \times \R^d \to \R^d$ be a random vector field, that determines a random tube $Q_\Vel(\omega)$, for any $\omega \in \Omega$. In the uniformly bounded case we assume the existence of a hold-all domain, i.e. we assume that  there exists a bounded open set $B$ such that $Q_\Vel(\omega)$ remains in $(0,\tau) \times B$, and that the velocity field is defined on this domain $B$, and not on the whole space $\R^d$.
How much the set $B$ varies from $D_0$, depends on how big the stochastic fluctuations are. Note that in the log-normal case $B=\R^d$, and we always write $B$ to cover both situations.

To a vector field $\Vel(\omega)$ we can associate its flow $T_\Vel(\omega)$. First, for fixed $\omega \in \Omega$ and $X \in \overline{D}_0 $ we consider the path-wise solution $x_{_\Vel}(\omega, \cdot, X)$ of the ODE
\begin{eqnarray}
\frac{d x_{_\Vel}}{dt}(\omega,t,X) &=& V(\omega, t, x_{_\Vel}(\omega,t,X)) \quad t \in[0,\tau] \\
x_{_\Vel}(\omega,0,X) &=& X.
\end{eqnarray}
For fixed $t$ and $X$, by Fubini's theorem, $\omega \mapsto x_{_\Vel}(\omega,t,X)$ is a measurable function.

Then, for any $\omega \in \Omega$ and $t \in [0,\tau]$, we consider the transformation
\begin{align*}
T_\Vel(\omega, t)&: \overline{B} \to \overline{B} \\
X \mapsto T_\Vel(\omega, t)(X) &:= x_{_\Vel}(\omega, t, X).
\end{align*}
We denote the mapping $(\omega, t, X) \mapsto  T_\Vel(\omega, t)(X)$ by $T_\Vel$. For brevity, and when there is no danger of confusion, we do  not write the associate vector field $\Vel$ but  we just write $T_t(\omega,X) \equiv T_\Vel(\omega, t)(X)$.  The measurability of $x$ implies the measurability of $\omega \mapsto T_t(\omega,X)$, for fixed $t$ and $X$.

Now, to a sufficiently smooth vector field $\Vel(\omega )$ we  associate a tube $Q_{\Vel,\tau}(\omega)$ defined by
\[
Q_{\Vel,\tau}(\omega) : = \bigcup_{t \in (0,\tau)} D_t(\omega)\times \{ t \} \qquad Q_0(\omega) := D_0,
\]
where $D_t(\omega):= T_t(\omega)(D_0)$. Similarly as for the flow, we  use the notation  $Q_\Vel(\omega) \equiv Q(\omega)$.  According to this notation,  one should consider  Definition \ref{def:L_X} of Bochner type spaces in path-wise sense for $T_t(\omega)$ instead of $\phi_t$, i.e. as already announced, $T_t(\omega)$ is a plain push-forward.  

\begin{remark}
Conversely, given a sample $\omega$ and a random tube $Q(\omega)$ with enough smoothness of  a lateral boundary that ensures the existence of the outward normal, we can associate to $Q(\omega)$ a random smooth vector field $\Vel(\omega)$ whose associated flow satisfies $ T_\Vel(\omega,t)(D_0)=D_t(\omega) \subset \R^d, \forall t \in [0,\tau]$.
\end{remark}

The relation between the regularity of the velocity field $\Vel(\omega)$ and the regularity of its associated flow $T_t(\omega)$ has been investigated using the general theory of shape calculus (for general results see for example \cite[Ch 4, Th 5.1]{DZ}). Here we  state weaker results that are sufficient for our analysis. These results are  presented in \cite[Proposition 2.1, Proposition 2.2]{DZEul} and \cite{DzTube}. First, let us state the assumptions about the velocity field.

\begin{ass}\label{ass:V}
The velocity field satisfies the following regularity assumption
\begin{equation}\label{V1general}
\Vel(\omega) \in C([0,\tau], W^{k,\infty}(B, \R^d)) \quad \text{ for a.e. } \omega \text{ and } k \geq 1.
\end{equation}
\end{ass}

\begin{ass}\label{ass:V1} For the unit outward normal field $n_B \in \R^d$ to $B$ we assume
\begin{equation}\label{V2general}
\Vel(\omega, t) \cdot n_B = 0 \quad \text{ on } \partial B, \quad \text{ for a.e. } \omega.
\end{equation}
\end{ass}

The assumption (\ref{V2general})  ensures that the transformation $T_t$ is one-to-one homeomorphism which maps $\overline{B}$ to  $\overline{B}$ (cf. \cite[pp. 87–-88]{Dug}). In particular, it maps the interior points onto interior points and the boundary points onto
boundary points. Thus, for every $t \in [0,\tau]$ we can consider the transformation $T_\Vel(t)^{-1} \equiv T_t^{-1} : \overline{B} \to \R^n$. Note that $T_t^{-1}$ is the flow at $s=t$ of the velocity filed $\tilde{\Vel}_t(s):= - \Vel(t-s)$. In the log-normal setting,  we assume  that the velocity field $\Vel$ is defined on the whole $\R^d$, this could be assumed also in the uniformly bounded setting. In this case the assumption (\ref{V2general}) is not needed and the analogue regularity results hold for the flow, see \cite[Theorem 4.1]{DZ}. 

\begin{remark}
Instead of (\ref{V2general}), one can make a more general assumption that $\pm \Vel(\omega,t,x)$ belongs to a so-called s Bouligand’s contingent cone. For more details see \cite[Ch. 5]{DZ}. 
\end{remark}

The following lemma  \cite[Proposition 2.1, 2.2]{DZEul} states the regularity results about the flow and will be used in a path-wise sense.

 \begin{lemma}
 Let  Assumption \ref{ass:V} hold. Then there exists a unique associated flow $T_\Vel$ that is a solution of
 \begin{equation}\label{flow_rel}
 \frac{d}{dt}T(t,\cdot) = \Vel(t,T(t, \cdot)), \quad T(0) = Id.
 \end{equation}
 such that  
 \begin{equation*}
 T_\Vel \in C^1([0,\tau],W^{k-1,\infty}({B},\R^d)) \cap C([0,\tau], W^{k,\infty}(B, \R^d)).
   \end{equation*}
 Moreover, 
 \[
 T^{-1}_\Vel \in C([0,\tau], W^{k,\infty}(B, \R^d)).
 \]
 \end{lemma}

For our analysis we need more regularity of the inverse transformation $T_t^{-1}$, i.e. the same as for $T_t$. Utilizing the implicit function theorem, better regularity result for $T_t^{-1}$ can be obtained on some subinterval $[0,\tau' ]$, see \cite[Proposition 2.2]{DzTube}.

\begin{lemma}\label{regTinv}
There exists $\tau' \in (0,\tau]$ such that $T_\Vel^{-1} \in C^1([0,\tau'],W^{k-1,\infty}({B},\R^d)) .$
\end{lemma}

Observe that in our setting we  consider Lemma \ref{regTinv} path-wise. Thus, for every $\omega$ there exists $\tau'(\omega) \in (0,\tau]$. For this reason we need to make an additional assumption to avoid that $\tau'(\omega)$ converges to zero. From now on, without loss of generality, we assume the existence of a deterministic constant $\tau_0$ such that 
\[
0< \tau_0 \leq \tau'(\omega) \leq \tau \quad \forall \omega
\]
and consider our problem on the time interval $[0,\tau_0]$. By abuse of notation, we continue to write $\tau$ for $\tau_0$. Hence, we have that
 \begin{equation}\label{regT}
 T_\Vel,  T^{-1}_\Vel \in C^1([0,\tau],W^{k-1,\infty}({B},\R^d)) \cap C([0,\tau], W^{k,\infty}(B, \R^d)).
   \end{equation}

Assuming that $\overline{B}$ has enough regular shape, such as bounded, open, path-connected and locally Lipschitz subset of $\R^d$, from \cite[Ch 2, Th 2.6]{DZ}, we infer
$$W^{k+1,\infty}(B, \R^d) ) = C^{k,1}(\overline{B}, \R^d) \quad \text{ and } \quad C^{k,1}(\overline{B}, \R^d) \hookrightarrow C^k(\overline{B}, \R^d).$$
In particular, in for our setting it is sufficient to assume that $k=2$ 
and consider the following regularity assumption:

\begin{ass}\label{ass:RV}
The velocity field satisfies the following regularity assumptions
\begin{equation}\label{V1}
\Vel(\omega) \in C([0,\tau], C^2(\overline{B}, \R^d)) \quad \text{ for a.e. } \omega.
\end{equation}
\end{ass}

\begin{ass}\label{ass:URV}
\begin{equation}\label{V2}
\Vel(\omega, t) \cdot n_B = 0 \quad \text{ on } \partial B \quad \text{ for a.e. } \omega.
\end{equation}
\end{ass}
Then, according to  \ref{regT},
we obtain the following regularity of the associated flow and its inverse
\begin{equation}\label{T}
T(\omega), T^{-1}(\omega) \in C^1([0,\tau], C(\overline{B},\R^d))  \cap C([0,\tau], C^2(\overline{B},\R^d))  .
\end{equation}

\begin{remark}
In the literature, a standard assumption for non-cylindrical problems is a monotone or regular (Lipschitz) variation of the domain $D_t$.
The weaker assumptions on time-regularity of the boundary are considered in \cite{BG} . Namely, the authors assume only the H\"older regularity for the variation of the domains. The motivating example for this kind of assumption is a  stochastic evolution problem in the whole space $\R^d$. 

Moreover, note that if $D_0$ is a Lipschitz domain and the transformation $T$ is bijective and bi-Lipschitz, then the transformed domain $D_t$ is not necessarily Lipschitz, see Zerner's example. However, bijective bi-Lipschitz transformations that are also $C^1$-diffeomorphisms  of the space do preserve the class of bounded Lipschitz domains \cite[Thm. 4.1]{lipschDomain}.
\end{remark}

In view of Assumption \ref{ass:RV}, spatial domains $D_t(\omega)$  in $\R^d$ are obtained from a base domain $D_0$ by a $C^2$-diffeomorphism, which is continuously differentiable in the time variable.
The $C^1$ dependence in time indicates that we do not have an overly rough evolution in time, and $C^2$ regularity in space means that topological properties are preserved along time.  

As a consequence of (\ref{T}), the following path-wise bound holds
\[
\|T(\omega)\|_{C([0,\tau], C^2(\overline{B},\R^d))}, \|T^{-1}(\omega)\|_{C([0,\tau], C^2(\overline{B},\R^d))}\leq C_T(\omega) \quad \text{ for a.e. } \omega.
\]

Let $DT_t(\omega)$ and $DT_t^{-1}(\omega)$ denote the Jacobian matrices of $T_t(\omega)$ and $T_t^{-1}(\omega)$, respectively. From 
(\ref{T}), we infer
\begin{equation}\label{DTreg}
DT(\omega), DT^{-1}(\omega) \in C^1([0,\tau], C^1(\overline{B},\R^d)) \quad \text{and} \quad
\frac{d}{dt}T(\omega) \in C([0,\tau], C(\overline{B},\R^d)) \quad \text{for a.e. } \omega.
\end{equation}

As a consequence we have
\begin{align}
\|DT(\omega)\|_{C([0,\tau], C^{1}(\overline{B},\R^d))}, \|DT^{-1}(\omega)\|_{C([0,\tau], C^{1}(\overline{B},\R^d))} \leq C_D(\omega) \label{DT}
\\
 \quad \text{ and } \,\, \| \frac{d}{dt} T(\omega) \|_{C([o,\tau], C(\overline{B}, \R^d))} \leq C_t(\omega) \quad  \text{ for a.e. } \omega. \label{tT}
\end{align} 
Since for the operator norm $\| \cdot \|$ of any square matrix $M$, it holds $\| M M^\tau \| = \| M^\tau M \| = \| M \|^2$, then from (\ref{DT}) we conclude 
\begin{equation}
\max_{t,x}\| DT_t(\omega,x) DT_t^\top(\omega,x) \| =\max_{x,t} \| DT_t^\top(\omega,x) DT_t(\omega,x) \| \leq C_D^2(\omega)  \quad \text{ for a.e. } \omega, \label{Mbound}
\end{equation}
and the analogue holds for the inverse Jacobian matrix.
Moreover, let $J_t(\omega) := \det(DT_t(\omega))$ and $J_t^{-1}(\omega) := \det(DT^{-1}_t(\omega))$. Since $J_t(\omega)$ does not vanish, $J_0(\omega)=1$, and because it is continuous, it follows that $J_t(\omega)>0$, a.e. and the same holds for its inverse.
From (\ref{DT}) we conclude 
\begin{equation}\label{Jreg}
J_{(\cdot)}(\omega), J_{(\cdot)}^{-1}(\omega) \in C^1([0,\tau], C^1(\overline{B},\R)) \quad \text{ for a.e. } \omega.
\end{equation}
The previous result implies that the gradient of the inverse Jacobian is bounded path-wise
\begin{equation}\label{ass:gradJ}
\| \nabla_x J_t^{-1} (\omega, x) \|_{\R^d} \leq C_J(\omega).
\end{equation}

\begin{remark}
Since $(M^\top)^{-1} = (M^{-1})^\top, M \in \R^{d \times d}$, inverse and transpose operations commute,  we will just write $M^\top$ for transpose and $M^{-\top}$ for its inverse.
\end{remark}

Furthermore, let $\sigma_i(\omega) = \sigma_i (DT_t(\omega,x)), i=1, \dots, d $  denote the singular values of the Jacobian matrix, i.e. the square root of eigenvalues of the matrix $DT_t DT_t^\top$ or equivalently, the matrix $DT_t^\top \, DT_t.$ If we consider a matrix which has continuous functions as entries, it follows that its eigenvalues  are also 
 continuous functions (see \cite{Zed}). This argument is based on the fact that the eigenvalues are roots of the characteristic polynomial and roots of any polynomial are continuous functions of its coefficients.  As the coefficients of the characteristic polynomial depend continuously on the entries of the matrix and singular values are the square roots of eigenvalues of $DT_t \,DT_t^\top$, it follows
 \[
 \sigma_i(\omega) \in C([0,\tau], C(\overline{B},\R)).
 \]

Thus, for every $i$, $\sigma_i(\omega)$ achieves the minimum and maximum on $[0,\tau ] \times \overline{B}$ and from (\ref{DT}) it follows
\begin{equation}
0 < \underline{\sigma}(\omega) \leq \min_{x,t} \{ \sigma_i(\omega,t,x) \}  \leq \max_{x,t} \{ \sigma_i(\omega,t,x) \} \leq \overline{\sigma}(\omega) < \infty \label{bdd:SVJ}
\end{equation}
for $ \overline{\sigma}(\omega)  := C_D(\omega)$, $\underline{\sigma}(\omega): = C_D^{-1}(\omega)$, 
for every $i=1,\dots,d$ and a.e. $\omega$.
Since $J(\omega) = \prod_{i=1}^n \sigma_i(\omega)$, the bound (\ref{bdd:SVJ}) implies 
\begin{equation}
0 < \underline{\sigma}^n(\omega) \leq J_t(\omega,x) \leq \overline{\sigma}^n(\omega) < \infty \quad \text{for every } x \in \overline{B}, t \in [0,\tau] \text{ and for a.e.} \omega. \label{bdd:J}
\end{equation}
The analogue reciprocal bounds hold for the $J_t^{-1}$. 

Note that all previous bounds are in a path-wise sense and are a direct consequence of a regularity result (\ref{T}). In the uniformly bounded case, we will in addition assume that these bounds $C_i(\omega)$ can be uniformly bounded in $\omega$. However, for the general setting we only need to assume that these constants are random variables with finite moments of desired order.
\begin{ass}\label{MbleConst}
Assume that $C_T, C_D, C_t, C_J$ belong to $L^p(\Omega)$ for every $p \in [1, \infty)$.
\end{ass}

\subsection{Unbounded transformation}\label{sec:unbbTr}

In this section we construct an example for which the Assumption \ref{MbleConst} is satisfied, but the constants are not uniformly bounded in $\omega$. This example is closely related to the exponential random field, which appears in log-normal growth models. However, since we need time and space regularity of the transformation stated in the previous subsection, we have to consider its regularization so that we can study the time derivative of the transformation $T$. Let $B_t$ be a standard $d$-dimensional Brownian motion and define its regularization by
\[
B_t^\epsilon := \int_0^t \rho _\epsilon (t-u) B_u du
\] 
where $\rho_\epsilon$ is a standard mollifier function and $C_\epsilon  := \| \rho \|_{C^1[0,\tau]}$, thus  $B_t^\epsilon$ belongs to $C^1([0,\tau], \R^d)$ .  Now we define the transformation 
\begin{equation}\label{expTr}
T_t(\omega, x) := x e^{B_t^\epsilon(\omega)} \quad x \in \mathbb{S}^1 \equiv D_0
\end{equation}
and the velocity transformation is given by $\Vel_t(\omega, x) = x e^{B_t^\epsilon(\omega)} \dot{B}_t^\epsilon(\omega)$. Hence, by definition the regularity assumptions concerning the space and time are full-filled, and thus path-wise bounds hold. We want to prove that these bounds belong to $L^p(\Omega)$, for $p \in [1,\infty)$, and that uniform bounds don't hold. 

\begin{lemma}\label{unboundedExample}
Let $T$ be defined by (\ref{expTr}), then the following holds
\begin{enumerate}
\item[i)] $\max\limits_{t \in [0,\tau],x \in D_0} \| T(\omega, t,x) \| \leq C_1(\omega) \quad $ and 
$\quad \max\limits_{t \in [0,\tau],x \in D_0 } \| DT(\omega, t,x) \| \leq C_2(\omega) $  for a.e. $\omega$
\item[ii)] $\max\limits_{t \in [0,\tau],x \in \mathbb{S}^1 } \| \frac{d}{dt} T(\omega, t,x) \| \leq C_3(\omega) $  for a.e. $\omega$.
\end{enumerate}
Moreover,  it holds $C_i \in L^p(\Omega), i=1,2,3$ for every $p  \geq 1 $ and for $i=1,2$, constants $C_i$ are not in $L^\infty(\Omega)$.
\end{lemma}
\begin{proof}
As already commented, the existence of the constants follows directly from the regularity of the fields, i.e. their continuity on compact sets $[0,\tau]$ and $\mathbb{S}^1$. The main part is to prove the regularity of constants. We choose norm on $D_0$ s.t. $\max_{x \in D_0} \| T_t(\omega, x) \| = e^{B_t^\epsilon(\omega)}$ and we bound its $p-$moments.   
Utilizing the Minkowski integral inequality we arrive at
\begin{align}
C_T(\omega) := \mathbb{E} [\max\limits_{t \in [0,\tau]} e^{B_t^\epsilon}] &= \mathbb{E} [ \max_{t \in [0,\tau]} \sum_{n \geq 0} \frac{1}{n!} (p B_t^\epsilon)^n]  \nonumber \\
&\leq \sum_{n \geq 0} \frac{p^n C_\epsilon^n}{n!} \mathbb{E} [\int_0^\tau |B_u | du] ^n 
\leq \sum_{n \geq 0} \frac{p^n C_\epsilon^n}{n!}( \int_0^s (\mathbb{E}[ |B_u|^n])^{1/n} ds )^n.  \label{lastTerm}
\end{align}
In order to bound the last term in (\ref{lastTerm}) we exploit the formula for raw absolute moments of a random variable and Stirling's formula, which yields to
\begin{equation}
\mathbb{E} [\max_{t \in [0,\tau]} e^{B_t^\epsilon}] \leq \sum_{n \geq 0} \frac{p^n C_\epsilon^n}{n!} \frac{1}{2^{3n}} C (n-1)^{n/2} e^{(1-n)/2} \tau^{2n} = \frac{c}{\sqrt{e}} \sum_{n \geq 0} C_{p,\epsilon}^n \frac{(n-1)^{n/2}}{2^{3n}n!} 
\end{equation}
where $C_{p, \epsilon} := \frac{pC_\epsilon}{\sqrt{e}} \tau^2$. Ration test implies the absolute convergence of the last series, and thus we showed that $C_T$ is bounded by a convergent series, it is also measurable as a maximum of a measurable function, and hence $C_T \in L^p(\Omega)$, for any $p\geq 1$.

For a lower bound, we use the formula for the expectation of exponential of a Gaussian random variable $p B_\tau^\epsilon$
\[
\mathbb{E} [ \max_{t\in [0,\tau],x\in D_0} \| T_t(\omega, x)] \geq \mathbb{E} e^{p B_\tau^\epsilon } = e^{\frac{\sigma^2 p^2}{2} }.
\]
Hence, 
\[
0 < \sqrt[p] {e^{\frac{\sigma^2 p^2}{2} }} \leq  \max_{t\in [0,\tau],x\in D_0} \| T_t( x) \|_{L^p(\Omega)} \leq C_T < \infty,
\]
i.e. $C_1 \in L^p(\Omega)$, for any $p \in [1, \infty)$, but as $p$ tends to infinity, the left hand side also tends to infinity, thus $C_1$ does not belong  to $L^\infty$, which completes the proof of i). Since $DT(\omega, t,x) = e^{B_t^\epsilon(\omega)} \text{Id}$, the second statement ii) follows in the same way as the previous one.

The first step in proving iii) is to apply the Cauchy-Schwartz inequality which implies
\begin{equation}\label{dtBound}
\mathbb{E} \left(  \max_{t\in [0,\tau],x\in D_0} \frac{d}{dt} \| T(t,x) \| \right)^p \leq \sqrt{\mathbb{E} (\max_{t \in [0,\tau]} e^{2pB_t^\epsilon})} \sqrt{\mathbb{E} (\max_{t \in [0,\tau]} |\dot{B}_t^\epsilon|^{2p}) }.
\end{equation}
For the first square root in (\ref{dtBound}) we use the same bound from i). For the second square root, we utilize again the bounds for the raw bounds of the moments of a random variable. Since
\[
\dot{B}_t^\epsilon = B_t \rho(0) + \int_0^t \rho'_\epsilon(t-u) B_u du,
\]
the binomial formula implies
\[
\mathbb{E} (\max_{t \in [0,\tau]} |\dot{B}_t^\epsilon|^{2p}) \leq 2^{p-1}  C_\epsilon^p \left(\mathbb{E}[\max_{t \in [0,\tau]} |B_t|^{2p}] + \mathbb{E} [\int_0^\tau |B_u|^{2p} du] \right),
\]
for simplicity of computations, we prove first for the even p, i.e. we use $2p$, and the analogue holds for $2p+1$. 
To bound the first term in the last sum, we exploit the Doob's inequality and arrive at
\[
\mathbb{E}[\max_t |B_t|^{2p}] \leq \left( \frac{2p}{2p-1} \right)^{2p} 2^p \tau^{2p} \frac{\Gamma(\frac{2p+1}{2})}{\sqrt{\pi}} =: C(p) \in L^p(\Omega)
\] 
for foxed $p \in [1,\infty)$. For the second term in the sum it holds
\[
 \mathbb{E} [\int_0^\tau |B_u|^{2p} du] = \int_0^\tau e^{2p^2u^2}du \leq \frac{\tau}{2p^2} e^{\frac{\tau}{p\sqrt{2}}} .
\]
Hence, we proved
\[
\| \max_{t \in [0,\tau]} |\dot{B}_t^\epsilon| \|_{L^{2p}(\Omega)} \leq C(p)^{1/2p} 2^{\frac{2p-1}{2p}} \tilde{C}_\epsilon^{1/2p} (C(p) + \frac{\tau}{2p^2} e^{\frac{\tau}{p\sqrt{2}}})^{1/(2p)}<\infty
\]
for fixed $p$. All together we have
\[
0 \leq \max_{t\in [0,\tau],x\in \mathbb{S}^1} \left\| \frac{d}{dt} T(t,x,\omega) \right\| \leq C_3(\omega) < \infty \quad \forall p <\infty
\]
where $C_3 \in L^p(\Omega)$. Note however, that we do not have the below bound away from zero and we do not prove that we do not have uniform bound for the $C_3$. Since we do not need these results for our example, and they are not straightforward,  we will not address this question in this work.  
\end{proof}
Since $J_t =  e^{B_t^\epsilon}\text{Id}$, it follows that $\nabla_x J_t^{-1} = 0$, thus $C_J = 0.$
From Lemma \ref{unboundedExample}, it follows that we have constructed the transformation $T_t$ that does satisfy the required regularity assumptions and all the bounds are random variables, however, this random field is not uniformly bounded in $\omega$ and thus does not fit into the setting of Subsection \ref{sec:UnifBdd} and is considered separately in the Subsection \ref{sec:lognormal}.  

\section{Heat equation on a random cylindrical domain}

As a standard example of a linear parabolic PDE, we consider the following initial boundary value problem for the heat equation in the non-cylindrical domain $Q(\omega)$
\begin{equation}\label{heateq}
\begin{aligned}
u' - \Delta u &= f \quad \text{in } Q(\omega) \\
u &= 0 \quad \text{ on} \cup_{t \in (0,\tau)} \partial D_t(\omega) \times \{ t\} \\
u(\omega,x, 0) &= u_0(x,\omega) \quad x \in D_0.
\end{aligned}
\end{equation}
We assume that the initial domain $D_0$ is deterministic and $u'$ is a weak time derivative.
\begin{remark}
The general form of the point-wise conservation law on an evolving flat domain $D_t$, derived in \cite{DE07}, is given by
\[
\partial^\bullet u + u \nabla \cdot \Vel + \nabla \cdot q =0
\] 
where $\Vel$ is the velocity of the evolution, $q$ is the flux and $\partial^\bullet$ is the material derivative. Taking in particular $q = -\nabla u - \Vel u$, we obtain the form (\ref{heateq}). Thus, although the material derivative does not explicitly appear in the formulation of the equation, as we have already commented, the material derivative is a natural notion for the derivative of a function defined on a moving domain. Thus, for the solution $u$, we will ask that its material derivative is in the appropriate space and we will use the solution space introduced in Section  \ref{sec:solution_space}.
\end{remark}

As already announced, in order to treat the path-wise formulation on a random tube, we  define
\begin{equation}\label{path_spaces}
V(\omega,t) := H_0^1(D_t(\omega)) \qquad H(\omega,t):=  L^2(D_t(\omega)), \qquad \forall \omega \in \Omega, \, \forall t \in (0,\tau).
\end{equation}
Given random velocity filed $\Vel$ induces a random flow $T_\Vel$ and for every $\omega \in \Omega$ and by  (\ref{path_flow}) we  define the path-wise pullback operator 
$\phi_{-t}(\omega):  L^2(D_t(\omega)) \rightarrow L^2(D_0)$.

By Lemma \ref{compatibility},  the spaces $L^2_{V(\omega)}$, $L^2_{{V^*(\omega)}}$ and $L^2_{H(\omega)}$ are well-defined for every $\omega$. Moreover, identifying  $L^2_{{V^*(\omega)}}$ and $(L_{V(\omega)}^2)^*$ and exploiting the density of the space $L^2(0,\tau;{V(\omega)})$ in $L^2(0,\tau;{H(\omega)})$, Lemma \ref{isom_with_Boh_sp},  we obtain the following result.

\begin{lemma}
\begin{equation*}
L^2_{H_0^1(D_t(\omega))} \xhookrightarrow{} L^2_{L^2(D_t(\omega))} \xhookrightarrow{} L^2_{H^{-1}(D_t(\omega))}
\end{equation*}
is a Gelfand triple, for every $\omega \in \Omega$.
\end{lemma}
Assuming enough regularity for the initial data $f$ and $u_0$, we specify the weak path-wise formulation of the boundary value problem (\ref{heateq}).
The path-wise solution space is a special case of the space (\ref{subspace}) and is defined by 
\begin{equation}\label{def:path_reg_sol_space}
W(H_0^1(D_t(\omega), L^2(D_t(\omega)))  := \{ u \in L^2_{H_0^1(D_t(\omega)} \mid \partial^\bullet{u} \in L^2_{ L^2(D_t(\omega))} \} \quad \forall \omega \in \Omega.
\end{equation}
Note that it is possible to consider less regular setting, as we will do in the final result, and then consider the solution space $W(V(\omega, t), V^*(\omega, t))$.

\begin{problem}[Weak path-wise form of the heat equation  on $D_t(\omega)$] \label{prob:RWF1} $\qquad \qquad \qquad \qquad \qquad \qquad$
For any $\omega$, find $u(\omega) \in W(H_0^1(D_t(\omega)),L^2(D_t(\omega)))$ that point-wise a.e. satisfies the initial condition $u(0)=u_0 \in L^2(\Omega,H^1(D_0))$ and  
\begin{equation} \label{RDweak form}
 \int_{D_t(\omega)} \!\! \left( u'(\omega, t) \varphi + \left< \nabla u(\omega, t) , \nabla \varphi \right>_{\R^n}\right)= \int_{D_t(\omega)} f(\omega,t) \varphi 
 \end{equation}
for every $\varphi \in  H_0^1(D_t(\omega))$ and a.e.  $t \in [0,\tau]$.
\end{problem}

Since (\ref{heateq}) is posed on a random domain, we would like to show that the solution $u$ is also a random variable and that it has finite moments. However, since the domain is random, we have $u(\omega, t) \in D_t(\omega)$. Thus defining the expectation of $u$ or of the random domain is not straightforward. The notion of a stochastic process with a random domain has already been analysed (see \cite{DF} and references therein). The authors begin by defining what is meant by a random open convex set in a probabilistic setting  and then continue by explaining what a stochastic process with a random domain is. Moreover, in \cite{CKR}, the authors give a possible interpretation of the notions of noise and a random solution on time-varying domains. We believe that these ideas could be generalized to our setting, but  they will not be analysed in this work. Instead, motivated by the domain mapping method, we consider the pull-back of the problem (\ref{heateq}) on the fixed domain $D_0$ and study the solution $\hat{u}$ of the reformulated problem. We  first derive the  path-wise formulation for the function $\hat{u}$ that is equivalent to Problem \ref{prob:RWF1}. For the function $\hat{u}$ it makes sense to ask $\hat{u} \in \mathcal{W}(H_0^1(D_0), L^2(D_0))$ and it is clear what its expectation or other quantity of interest, are. Thus, exploiting the domain mapping method, we translate the PDE on the random non-cylindrical domain into a PDE with random coefficients on the fixed cylindrical domain $\hat{Q}:= D_0 \times [0,\tau]$. 

Let $\hat{u}(\omega): [0,\tau ] \times D_0 \to \R$ be defined by the plain pull-back transformation
\begin{equation}\label{transf}
\hat{u}(\omega,t,y) := u(\omega, t, T_t(\omega, y) ) \quad \text{for every } y \in D_0, t \in [0,\tau].
\end{equation}

\begin{lemma}[Formulae for transformed $\nabla$ and $\partial_t$]\label{CR}
Let $f(\omega)\! \in \!L^2_{V(\omega,t)} $ and $\hat{f}(\omega) : \hat{Q} \to \R$,  $\hat{f}(\omega, t, X) := f(\omega\!, t, T_t(\omega,  X))$, for every $\omega \in \Omega$. Then
\begin{align}
 \nabla_x f (\omega, t, T_t(\omega, y)) &= \, DT_t^{-\top}(\omega,y) \, \nabla_y \hat{f}(\omega,t,y)  \quad y \in D_0\label{gradCR}\\
f'(\omega, t, T_t(\omega,y)) &= \hat{f}'(\omega,t,y) -   V(\omega,t, T_t(y))    (DT_t^{-\top}(\omega,t,y)\nabla_y \hat{f}(\omega,t,y)) \quad y \in D_0.  \label{timeCR}
\end{align}
\end{lemma}
\begin{proof}
We do the proof in a path-wise regime. The identity (\ref{gradCR}) follows directly from the chain rule (see \cite[Proposition IX.6]{sepBrezis}) and definition (\ref{transf}).
Utilizing  once more  the chain rule for the derivative w.r.t. time,  the  relation (\ref{gradCR}), and (\ref{flow_rel}), we get
\begin{eqnarray*}
\hat{f}'(t,y) &=& f'(t,T_t(y)) + \, DT_t^{-\top}(t,y)\nabla_y \hat{f}(\omega,t,y)) \cdot \frac{\partial T}{\partial t}(t,y) \\
&=& f'(t,T_t(y)) + \, (\,DT_t^{-\top}(\omega,t,y)\nabla_y \hat{f}(\omega,t,y))\cdot  V(t,T_t(y))
\end{eqnarray*}
which implies the relation (\ref{timeCR}). 
\end{proof}

\begin{problem}[Weak path-wise form of the heat equation on $D_0$] \label{prob:RWF2}
 For every $\omega$, find  $\hat{u}(\omega) \in W(H_0^1(D_0), L^2(D_0))$ that point-wise a.e. satisfies the initial condition $u(0)=u_0$, $u_0 \in L^2(\Omega,H^1(D_0))$ and  
\begin{multline} \label{weak form on D}
 \int_{D_0} \ \left(\hat{u}'(t,y) - \left< \,DT_t^{-\top}(t,y) \nabla \hat{u}(t,y), V(t,T_t(y))\right>_{\R^d} \right)J_t(y) \hat{\varphi}(y) \\
 + \left<A(t, y) \nabla \hat{u}(t,y), \nabla \hat{\varphi}(y) \right>_{\R^d} dy = 
 \int_{D_0}  \hat{f} (t,y)\hat{\varphi}(y) J_t(y) dy
 \end{multline}

for every $\hat{\varphi} \in  H_0^1(D_0)$ and a.e.  $t \in [0,\tau]$,
where
\begin{equation}\label{matrixA}
A(\omega,t,y ) = J_t(\omega,y) D T_t^{-1}(\omega,y) \, D T_t^\top(\omega,y)^{-1} \quad y \in D_0.
\end{equation}
\end{problem}

\begin{lemma}[Path-wise formulations on $Q_T(\omega)$ and $\hat{Q}_T$] 

Letting $f \in L^2_{L^2(\Omega, L^2( D_t(\omega)))}$,  the following are equivalent:

i) $u(\omega)$ is a path-wise weak solution to Problem \ref{prob:RWF1}

ii) $\hat{u}(\omega)$ is a path-wise weak solution to Problem \ref{prob:RWF2}.
\end{lemma}

\begin{proof}
Let us assume that i) holds. From the substitution rule $x = T_t(y)$ and Lemma \ref{CR}, we obtain
\begin{eqnarray*}
\int_{D_0}u'(t,T_t(y)) \varphi(t,T_t(y)) J_t(y) dy + \int_{D_0} \nabla u (t,T_t(y)) \cdot \nabla \varphi(t,T_t(y)) J_t(y) dy = \nonumber \\
\int_{D_0} \left( \hat{u}' (t,y) - \,DT_t^{-\top}(t,y) \nabla \hat{u}(t,y) \cdot V(t,T_t(y)) \right) \hat{\varphi}(t,y) J_t(y) dy + \nonumber \\
\int_{D_0} \left< \,  DT_t^{-\top}(y)\nabla \hat{u}(t,y) , DT_t^{-\top}(y) \nabla \hat{\varphi}(t,y)\right>_{\R^d} J_t(y)dy =  \label{chainrule} \\
\int_{D_0} \hat{f}(t,y)\hat{ \varphi}(t,y) J_t(y) dy. \nonumber
\end{eqnarray*}

Since
\[
\int_{D_0}\!\! \left< \,  DT_t^{-\top}(y)\nabla \hat{u}(t,y) , DT_t^{-\top}(y) \nabla \hat{\varphi}(t,y)\right>_{\R^d} J_t(y)dy = 
\int_{D_0}  \!\!\!\left<A(t, y) \nabla \hat{u}(t,y), \nabla \hat{\varphi}(y) \right>_{\R^d} dy, 
\]
where the matrix $A$ is defined by (\ref{matrixA}), it follows that that $\hat{u}$ is a path-wise weak solution of Problem \ref{prob:RWF2}. 
The proof of implication ii) $\Rightarrow$ i) is similar.
\end{proof}
Note that according to Lemma \ref{isom_with_Boh_sp}, it holds
\[
u(\omega) \in L^2_{H_0^1(D_t(\omega))} \Leftrightarrow \hat{u}(\omega) \in L^2(0,\tau; H_0^1(D_0)) \quad \text{ for a.e. } \omega.
\]

For the mean-weak formulation on the fixed domain we consider
\begin{equation}\label{our_spaces}
\mathcal{V} := L^2(\Omega, H_0^1(D_0)) \qquad \mathcal{H}:= L^2(\Omega, L^2(D_0)).
\end{equation}
Utilizing the tensor structure  (\ref{tensor}) of such defined $\mathcal{V}$ and $\mathcal{H}$, and the density argument, we obtain that $\mathcal{V} \xhookrightarrow[i]{} \mathcal{H} \cong \mathcal{H}^*\xhookrightarrow[i']{} \mathcal{V}^*$ is a Gelfand triple. 

\begin{remark}
The spaces $H_0^1(D_0)$ and $H_0^1(D_t(\omega))$ are isomorphic due to the isomorphism $\eta \mapsto \eta \circ T_t(\omega)^{-1}$. This implies that the space of test functions in the weak formulation is independent of $\omega$. For more details see \cite[Lemma 2.2]{HPS}.
\end{remark}

\subsection{Uniformly bounded transformation - well-posedness of  the transformed equation}\label{sec:UnifBdd}

In this subsection we make additional assumptions concerning uniformity in $\omega$ for   the bounds $C_j(\omega)$ of the transformation $T$ and its derivatives. We  use the same notation for uniform constants as we did for the random variables in Section \ref{sec:randomTubes}.  To ensure the uniform bound and the coercivity of the bilinear form that will be considered, we suppose to have the uniform bound of the norm.

\begin{ass}\label{UUC}
We assume that there exists a constant $C_T>0$ such that 
\[
\|T(\omega)\|_{C([0,\tau], C^2(\overline{B},\R^d))}, \|T^{-1}(\omega)\|_{C([0,\tau], C^2(\overline{B},\R^d))} \leq C_T \quad \text{ for a.e. } \omega.
\]
\end{ass}

In addition, we  assume a uniform bound for the Jacobian matrices, time derivative and  gradient of the inverse Jacobian. The regularity result (\ref{Jreg}) implies they are bounded, but not that these bounds are uniform in $\omega$. 
\begin{ass}\label{UBAss}
There exist constants $C_D, C_t, C_J > 0$ s.t.
\begin{align}
\|DT(\omega)\|_{C([0,\tau], C^{1}(\overline{B},\R^d))}, \|DT^{-1}(\omega)\|_{C([0,\tau], C^{1}(\overline{B},\R^d))} &\leq C_D \label{UDT}
\\
 \quad  \| \frac{d}{dt} T(\omega) \|_{C([o,\tau], C(\overline{B}, \R^d))} &\leq C_t \label{UtT} \\
 \| \nabla_x J_t^{-1} (\omega, x) \|_{\R^d} &\leq C_J.
 \quad  \text{ for a.e. } \omega. \label{Uass:gradJ}
\end{align} 
\end{ass}

The independence on $\omega$ of minimal and maximal values of  $\sigma_i(\omega)$, for any $i$ follows from (\ref{DT}). To see this, recall that the Rayleigh quotient
and the definition of the singular value imply
\[
\sigma_i(\omega, x, t) \leq \max_{\|y \|_{\R^d}=1} \| DT_t(\omega,x) y \|_{\R^d}.
\]

Thus, for $ \overline{\sigma}  := C_D$, $\underline{\sigma}: = C_D^{-1}$, 
 every $i=1,\dots,d$ and a.e. $\omega$  we have 

\begin{equation}
0 < \underline{\sigma} \leq \min_{x,t} \{ \sigma_i(\omega,t,x) \}  \leq \max_{x,t} \{ \sigma_i(\omega,t,x) \} \leq \overline{\sigma} < \infty. \label{Ubdd:SVJ}
\end{equation}
Since $J(\omega) = \prod_{i=1}^n \sigma_i(\omega)$, the bound (\ref{bdd:SVJ}) implies the uniform bound for the determinant of the Jacobian, i.e. for a.e. $\omega$ it holds
\begin{equation}
0 < \underline{\sigma}^n \leq J_t(\omega,x) \leq \overline{\sigma}^n < \infty \quad \text{for every } x \in \overline{B}, t \in [0,\tau]. \label{Ubdd:J}
\end{equation}
The analogue reciprocal bounds hold for the $J_t^{-1}$. 

After stating all the assumptions, we want to write  (\ref{weak form on D}) in a standard form, which is more convenient  to apply the general theory of well-posedness for parabolic PDEs presented in \cite{Wloka}, i.e. we remove the weight $J_t^{-1}$ in front of the time derivative $\hat{u}'$. This form we can achieve by testing the equation  (\ref{weak form on D}) with functions $\hat{\varphi}(t,y) = J_t^{-1}(y)\tilde{\varphi}(t,y)$. The spatial regularity of $J_t$ stated in  (\ref{Jreg}), implies 
$$\forall \hat{\varphi} \in H_0^1(D_0) \Leftrightarrow \forall \tilde{\varphi} \in H_0^1(D_0).$$
In this way we obtain the equivalent form of (\ref{weak form on D}) given by
\begin{multline}
 \int_{D_0} \ \left(\hat{u}'(t,y) - \left< \,DT_t^{-\top}(t,y) \nabla \hat{u}(t,y), V(t,T_t(y))\right>_{\R^d} \right)\tilde{\varphi}(y) \\
 + \left<A(t, y) \nabla \hat{u}(t,y), \nabla (J_t^{-1}(y)\tilde{\varphi}(y)) \right>_{\R^d} dy = 
 \int_{D_0}  \hat{f} (t,y)\hat{\varphi}(y) dy,
 \end{multline}
for all   $\tilde{\varphi} \in H_0^1(D_0)$. Utilizing the product rule for the gradient and symmetry of the matrix $A$, we arrive at the equivalent weak path-wise form of the heat equation:

\begin{problem}[Weak path-wise form of the heat equation  on $D_0$] \label{newprob:RWF3}
 For every $\omega$, find  $\hat{u}(\omega) \in W(H_0^1(D_0), L^2(D_0))$ that point-wise a.e. satisfies the initial condition $u(0)=u_0$, $u_0 \in L^2(\Omega,H_0^1(D_0))$ and  
\begin{multline} \label{new weak form on D}
 \int_{D_0} \ \left(\hat{u}'(t,y) + \left< A(t, y)  \nabla J_t^{-1}(y) - DT_t^{-1}(y) V(t,T_t(y)), \nabla \hat{u}(t,y) \right>_{\R^d} \right)  \tilde{\varphi}(y) \\
 + \left< DT_t^{-1}(y) DT_t^{-\top}(t,y) \nabla \tilde{\varphi} (t,y) ,  \nabla \hat{u}(t,y)\right>_{\R^d} dy =
 \int_{D_0}  \hat{f} (t,y) \tilde{\varphi}(y)  dy
 \end{multline}
for every $\tilde{\varphi} \in  H_0^1(D_0)$ and a.e.  $t \in [0,\tau]$.
\end{problem}
Observe that the partial integration and the fact  that a test function vanishes on the boundary $\partial D_0$ imply
\begin{multline*}
\int_{D_0} \left< \, DT_t^{-1}(y)  DT_t^{-\top}(y)\nabla \tilde{u}(t,y) , \nabla \hat{\varphi}(t,y)\right>_{\R^d}dy = \\
- \int_{D_0} \text{div}(DT_t^{-1}(y)  DT_t^{-\top}(y) \nabla \hat{u} (t,y)) \tilde{\varphi}(t,y)dy.
\end{multline*}
Let us comment on the boundary condition and initial condition. Since $T_0$ is the identity and $D_0$ is the deterministic initial domain, the initial condition stays the same:
 \[u(\omega, x,0)  = u_0(\omega,x) \Leftrightarrow \hat{u}(\omega,x,0) = u_0(\omega,x), \quad \forall x \in D_0,
  \]
for a.e. $ \omega \in \Omega$.  Moreover,  as the boundary of $\partial D_t(\omega)$ is mapped  to  $\partial D_0$, the reformulated boundary condition stays the same:
\begin{eqnarray*}
u(\omega, t,x) &=& 0 \quad \forall (x,t) \in \cup_{t \in (0,\tau)}\partial D_t(\omega)  \times \{t\} \Leftrightarrow \\ 
\hat{u}(\omega, t, y) & =& 0 \quad \forall (y,t) \in \partial D_0 \times (0,\tau) \quad 
\end{eqnarray*}
 \text{ for a.e. } $\omega \in \Omega$. 
Hence,  in the distribution sense, we are led to consider for a.e. $\omega$ 
\begin{eqnarray*}
\hat{u}' - \text{div}(J_t^{-1}A \nabla \hat{u}) + \left< \nabla \hat{u}, A \nabla J_t^{-1} - DT_t^{-1} V \circ T_t \right>_{\R^d} &=& \hat{f} \quad \text{ in } (0,\tau) \times D_0 \\
\hat{u}(\omega,x,t)  &=& 0 \quad \text{ on } \partial D_0 \times (0,\tau) \\
\hat{u}(\omega,x,0) &=& u_0(\omega,x) \quad \text{ on } D_0.
\end{eqnarray*}

Our goal is to show that $\hat{u}$ is a random variable and that it has finite moments, under suitable assumptions on the initial data. Thus, we  formulate a mean-weak formulation for $\hat{u}$. Furthermore, we  prove a more general result, when we have less regularity in the initial data. The regularity results can be obtained from the general theory on parabolic PDEs. In particular, assuming more regularity on $\hat{f}$ and $u_0$, we obtain better regularity of the time derivative of $\hat{u}$.

Observe that since $L^2(\Omega)$ is separable, utilizing tensor product isomorphisms (\ref{tensor}), we conclude
\begin{eqnarray*}
L^2(\Omega) \otimes L^2(0,\tau; H) &\cong& L^2(\Omega, L^2(0,\tau; H)) \cong L^2(\Omega \times (0,\tau); H) \\ &\cong& L^2(0,\tau; L^2(\Omega; H)) \cong L^2(0,\tau) \otimes L^2(\Omega, H)
\end{eqnarray*}
for any Hilbert space $H$. Thus, it holds
\[
L^2(\Omega) \otimes \mathcal{W}(H_0^1(D_0), H^{-1}(D_0)) \cong \mathcal{W}(L^2(\Omega,H_0^1(D_0)),L^2(\Omega,H^{-1}(D_0))),
\]
where $\mathcal{W}(L^2(\Omega,H_0^1(D_0)),L^2(\Omega,H^{-1}(D_0)))$ is a standard Bochner space defined by (\ref{def:stdSBsp}).

\begin{problem}[Mean-weak formulation on $D_0$] \label{prob:RMWF}
Find $\hat{u} \!\in \! \mathcal{W}_0(L^2(\Omega,\!H_0^1(D_0)\!),L^2(\Omega,\!H^{-1}(D_0)\!)\!) $ such that a.e. in $[0,\tau]$ it holds
\begin{eqnarray*} \label{mean-weak form}
&\int_\Omega \int_{D_0} \left< \hat{u}', \varphi \right>_{H^{-1}(D_0), H^1(D_0)} dy d \mathbb{P}+
 \int_\Omega \int_{D_0} \left< D T_t^{-1}(\omega,y) \, D T_t^{-\top}(\omega,y) \nabla \hat{u}, \nabla
\varphi \right>_{\R^d} dy d \mathbb{P}  + \\ 
&\int_\Omega \int_{D_0}  \left<A(\omega,t,y) \nabla J_t^{-1}(\omega,y) - DT_t^{-1}(\omega,y) V (t, T_t(y)), \nabla \hat{u} \right>_{\R^d} \varphi  dy d \mathbb{P} 
= \int_\Omega \int_{D_0}  \hat{f} \varphi dy d\mathbb{P}
 \end{eqnarray*}
for every $\varphi  \in  L^2(\Omega,H_0^1(D_0))$.
\end{problem}

\begin{theorem}\label{T1}
Let Assumptions \ref{ass:V}, \ref{ass:V1}, \ref{UUC} and \ref{UBAss}  hold and $f \in L^2_{L^2(\Omega,H^{-1}(D_0))}$. Then, there is a unique solution $\hat{u} \in  \mathcal{W} (L^2(\Omega, H_0^1(D_0)), L^2(\Omega, H^{-1}(D_0)))$ of Problem \ref{prob:RMWF} and we have the
a priori bound
\begin{equation}
\|\hat{u}\|_{ \mathcal{W} (L^2(\Omega, H_0^1(D_0)), L^2(\Omega, H^{-1}(D_0)))} \leq C \|f\|_ {L^2_{L^2(\Omega,H^{-1}(D_0))} }\label{RDapriori}
\end{equation}
with some deterministic constant $C > 0$.
\end{theorem}
\begin{proof}
For every $t \in [0,\tau]$ we introduce the bilinear form $a(t;\cdot,\cdot): \mathcal{V} \times \mathcal{V} \to \R$ by
\begin{equation}
a(t;\varphi,\psi) := \int_\Omega \int_{D_0} \!\!\left( \left< D T_t^{-1} \, D T_t^{-\top}  \nabla \varphi, \nabla
\psi \right>_{\R^n} \!\!\! +\!\! \left<A\nabla J_t^{-1} - DT_t^{-1} V \circ T_t , \nabla \varphi \right>_{\R^n} \!\! \psi \right) dx d\mathbb{P} . \label{def:a}
\end{equation}
We will prove that $a(t; \varphi, \psi)$ satisfies the following assumptions, which are necessary conditions for the well-posedness of the parabolic PDE stated in \cite[Theorem 26.1]{Wloka}.
\begin{enumerate}
\item[i)] $a(t; \varphi, \psi)$ is measurable on $[0,\tau]$, for fixed $\varphi, \psi \in \mathcal{V}$.
\item[ii)] There exists some $c> 0$, independent of $t$, such that
\begin{equation}\label{ass_one}
|a(t;\varphi, \psi)| \leq c \| \varphi \|_\mathcal{V} \| \psi \|_\mathcal{V} \quad \text{for all } t \in [0,\tau], \varphi, \psi \in \mathcal{V}.
\end{equation}
\item[iii)] There exist real $k_0, \alpha \geq 0$ independent of $t$ and $\varphi$, with
\begin{equation}\label{ass_two}
a(t; \varphi, \varphi) + k_0 \| \varphi \|_\mathcal{H}^2 \geq \alpha \| \varphi \|^2_\mathcal{V} \quad \text{ for all } t \in [0,\tau], \varphi \in \mathcal{V}.
\end{equation}
\end{enumerate}
 
i) Due to Assumptions \ref{ass:RV} and \ref{ass:URV}  and regularity results (\ref{DTreg}) and (\ref{Jreg}),  the integrand in the definition (\ref{def:a}) is $\mathcal{B}([0,\tau])$-measurable. Consequently, according to Fubini's theorem,  we obtain the Borel measurability on $[0,\tau]$ of the mapping
$t \mapsto a(t;\varphi, \psi)$
for fixed $\varphi, \psi \in \mathcal{V}$. 

ii) Applying the Cauchy-Schwartz inequality, we infer
\begin{equation}
\int_\Omega \int_{D_0} \!\!|\left< D T_t^{-1} \, D T_t^{-\top}  \nabla \varphi, \nabla \psi \right>_{\R^n}\!\! | \leq 
\int_\Omega \!\int_{D_0} \!\! \|  D T_t^{-1} \, D T_t^{-\top}  \nabla \varphi \|_{\R^d} \| \nabla \psi  \|_{\R^d} \leq 
C_1 \| \nabla \varphi \|_\mathcal{V} \|\nabla \psi \|_\mathcal{V}, \label{lmax_bound}
\end{equation}
where the last inequality follows from (\ref{Mbound}) and (\ref{UDT}), for $C_1 = \underline{\sigma}^2$. 

According to the triangular inequality we have
\begin{multline*}
 \| A(\omega, t, \cdot) \nabla J_t^{-1}(\omega,\cdot) - DT_t^{-1}(\omega,\cdot) V (t, T_t(\cdot))  \|_\infty  \leq \\
  \| A(\omega, t, \cdot) \nabla J_t^{-1}(\omega,\cdot) \|_\infty +\| DT_t^{-1}(\omega,\cdot) V (t, T_t(\cdot))  \|_\infty.
\end{multline*}
The uniform bound of the second term follows from (\ref{UDT}) and (\ref{UtT}). Concerning the first term, utilizing  Assumption \ref{Uass:gradJ} we get
\[
 \| A(\omega, t, \cdot) \nabla J_t^{-1}(\omega,\cdot) \|_\infty \leq C_{J} \|A\|_\infty.
\]
Moreover, from  (\ref{Ubdd:SVJ}) and (\ref{Ubdd:J}) we conclude
\[
\|A\|_\infty \leq \lambda_{\max}{A} \leq \overline{\sigma}^d \lambda_{\max}({D T_t^{-1} \, D T_t^{-\top}} )\leq \overline{\sigma}^d \underline{\sigma}^2,
\]
which yields to the bound
\begin{multline}
 \| A(\omega, t, \cdot) \nabla J_t^{-1}(\omega,\cdot) - DT_t^{-1}(\omega,\cdot) V (t, T_t(\cdot))  \|_\infty \\
  : =\max_{y \in D_0}   \| A(\omega, t, y) \nabla J_t^{-1}(\omega,y) - DT_t^{-1}(\omega,y) V (t, T_t(y))  \|_{\R^d} \leq C_2,\label{bound_ADVJ}
\end{multline}
for some $C_2>0$ independent of $t$. 

Utilizing the Cauchy-Schwartz inequality and \ref{bound_ADVJ} we infer
\begin{equation}\label{bound_second_term}
\begin{aligned}
&\int_\Omega \int_{D_0} |\left<\nabla \varphi(\omega,t,y), A(\omega, t, y) \nabla J_t^{-1}(\omega,y) - DT_t^{-1}(\omega,y) V (t, T_t(y)) \right>_{\R^d}| | \psi(\omega,t, y)|  \\
&\leq \int_\Omega \int_{D_0}  \| A(\omega, t, y) \nabla J_t^{-1}(\omega,y) - DT_t^{-1}(\omega,y) V (t, T_t(y))  \|_{\R^d} \|\nabla \varphi(\omega,t,y)\|_{\R^d} | \psi(\omega,t, y)|  \\
&\leq C_2 \|  |\nabla \varphi | \|_{H} \| \psi \|_H. 
\end{aligned}
\end{equation}
Finally, inequalities (\ref{lmax_bound}) and (\ref{bound_second_term}), ensure the condition ii).

iii) The bound (\ref{Ubdd:SVJ}) that implies the bound for the eigenvalue 
$\lambda_{\min}(D T_t^{-1} \, D T_t^{-\top}) \geq \frac{1}{\overline{\sigma^2}} =: C_3$. Exploiting this bound and the  Rayleigh quotient of the minimal eigenvalue of the symmetric matrix $ D T_t^{-1} \, D T_t^{-\top} $, we obtain
\begin{eqnarray*}
C_3 \| \nabla \varphi \|_H^2 &\leq& \int_\Omega \int_{D_0} \lambda_{\min} (D T_t^{-1} \, D T_t^{-\top})  \|\nabla \varphi \|^2_{\R^d} \\
&\leq& a(t; \varphi, \varphi) + \int_\Omega \int_{D_0} \| \nabla \varphi \|_{\R^d} \| DT_t^{-1} V \circ T_t - A \nabla J_t^{-1} \|_{\R^d} |\varphi| \\
&\leq&a(t; \varphi, \varphi) + C_2 \| \nabla \varphi \|_{H}\| \varphi \|_H \\
&\leq&a(t; \varphi, \varphi) + C_2  \left(2 \varepsilon \| \nabla \varphi \|^2_H + \frac{1}{2 \varepsilon} \| \varphi \|^2_H \right), 
\end{eqnarray*}
where we used Young's inequality in the last step. For small enough $\varepsilon >0$, we get
\[
(C_3 - 2 \varepsilon) \| \nabla \varphi \|_H^2 \leq a(t; \varphi,\varphi) + k_0 \|\varphi\|^2_H,
\]
for $k_0 := C_2 \frac{1}{2 \varepsilon}$. Applying  Poincare's inequality with the constant $C_P$ from
\[
\frac{C_3- 2\varepsilon }{1+ C^2_P} \| \varphi \|_V^2 \leq (C_3 - 2 \varepsilon) \| \nabla \varphi \|_H^2 \leq a(t; \varphi,\varphi) + k_0 \|\varphi\|^2_H,
\]
we conclude that iii) holds with $\alpha = \frac{C_3- 2\varepsilon }{1+ C^2_P}$.
 
After proving i), ii) and iii), the classical result  \cite[ Theorem 26.1]{Wloka} yields the existence and uniqueness of the solution $\hat{u}$ that satisfies an a priori bound (\ref{RDapriori}).
\end{proof}

\subsection{Path-wise bounded transformation - well-posedness of the transformed equation}\label{sec:lognormal}

In this subsection we want to show that even if we don't assume uniform bounds stated in Assumptions \ref{ass:RV} and \ref{ass:URV}, which is the case for example stated in Subsection \ref{sec:unbbTr}, we can still prove existence and uniqueness of the solution of the transformed equation on the cylindrical domain, i.e. Problem \ref{newprob:RWF3}. However, in this case we can't consider the mean-weak formulation and directly apply the general theory, but instead, we  prove the path-wise well-posedness and derive the path-wise a priori bound. In addition, we prove the measurability of the solution and $p-$bound of its moments $p \in [1, \infty)$.  

For the path-wise setting, we define
\[
\mathbb{V} := H_0^1(D_0) \quad \text{ and } \quad \mathbb{H} := L^2(D_0)
\]
which form the Gelfand triple. We consider a path-wise bilinear form $\mathbbm{a} (\omega, t; \cdot, \cdot):\mathbb{V} \times \mathbb{V} \to \R$ 
\begin{equation}
\mathbbm{a}(\omega, t; \varphi, \psi ) := 
 \int_{D_0} \!\left( \left< D T_t^{-1} \, D T_t^{-\top}  \nabla \varphi, \nabla
\psi \right>_{\R^n}  +\left<A\nabla J_t^{-1} - DT_t^{-1} V \circ T_t , \nabla \varphi \right>_{\R^n} \! \psi \right) dx  \label{def:patha}
\end{equation}
that is analogue to (\ref{def:a}). Assume for simplicity that  the initial data $u_0$ is deterministic. The stochastic case can be handled in an analogue way.

\begin{theorem}\label{T2}
Let Assumptions \ref{ass:V}, \ref{ass:V1}, \ref{MbleConst}  hold and $f \in L^2([0,\tau],L^2(\Omega, H^{-1}(D_0)))$. Then, there is a unique solution $\hat{u} \in  \mathcal{W} ( \mathbb{V}, \mathbb{V}^{-1})$ of Problem \ref{newprob:RWF3} and we have the
a priori bound
\begin{equation}
\|\hat{u}(\omega)\|_{ \mathcal{W} ( \mathbb{V}, \mathbb{V}^{-1})} \leq C(\omega) \|f(\omega)\|_ {L^2([0,\tau], H^{-1}(D_0)) }\label{PathRDapriori}
\end{equation}
for a.e. $\omega$ and where
\begin{equation}\label{COmega}
C(\omega) = 2\Big(C_L(\omega) \mathbb{C}(\omega)+1\Big) + \mathbb{C}(\omega)
\end{equation}
where $\mathbb{C}(\omega) := \max \Big\{ \frac{2(1+C_M)}{\alpha(\omega)}  , \frac{1}{\alpha^2(\omega)} \Big\}$ and $C_L(\omega)$ is a path-wise bound of the operator $L(\omega,t) v:= -\text{div}(J_t^{-1}A\nabla v) + \left< \nabla v, A \nabla J_t^{-1} - DT_t^{-1} V \circ T_t \right>_{\R^d} + k_0 Id$, for $k_0$ from G\r{a}rding's inequality.
Moreover,  for every $\omega$ the solution depends continuously on $(f(\omega),u_0)$ as a map from $L^2([0,\tau],L^2(\Omega, \mathbb{V}^*)) \times \mathbb{H}$ to $\mathcal{W}(\mathbb{V}, \mathbb{V}')$. 
\end{theorem}
\begin{proof}
The proof is a path-wise analogue to the proof of Theorem \ref{T1} and it is based on the application of the general result \cite[ Theorem 26.1]{Wloka}. However, unlike in the uniformly bounded case, the constants that appear depend on the sample $\omega$ and in order to prove $\hat{u} \in \mathcal{W}(\mathcal{V}, \mathcal{V}^* )$, we need to control those constants and show that $C(\omega)$ satisfies (\ref{COmega}).

The condition i) is full field for the same reason as in the uniform case. In the condition ii), we have that the bound in (\ref{ass_one}) satisfies
\begin{equation}\label{condII}
C_1(\omega) := \underline{\sigma}^2(\omega) [1 + C_J(\omega) \overline{\sigma}^d(\omega)] + C_D(\omega)C_t(\omega).
\end{equation}
For the condition iii) we obtained $\alpha = \frac{C_3(\omega) - 2 \epsilon(\omega)}{ 1 + C^2_P}$
where $C_P$ is a deterministic Poincare's constant that depends on the domain $D_0$ and $C_3(\omega) = \overline{\sigma}^{-2}$. We chose $\epsilon(\omega)$ small enough such that $C_3(\omega) - 2 \epsilon(\omega)>0$, since we can not uniformly bound $C_3(\omega)$ from below, we choose for example $\epsilon(\omega) = \frac{C_3(\omega)}{4}$. Hence, we have
\[
\mathbbm{a}(\omega, t; \varphi, \varphi) + k_0(\omega) \| \varphi \|_{\mathbb{H}}^2 \geq \alpha(\omega) \| \varphi \|^2_{\mathbb{V}}
\] 
where 
\begin{equation}\label{kOmega}
\alpha(\omega) = \frac{1}{2 \overline{\sigma}^2(\omega) (1+C_P^2)} \quad \text{ and }  \quad
k_0(\omega) = 2 \overline{\sigma}^2(\omega) [C_D(\omega) C_t(\omega) + C_J(\omega) \overline{\sigma}^d(\omega) \underline{\sigma}^2(\omega)].
\end{equation}
and 
\begin{equation}\label{CComega}
\mathbb{C}(\omega) := \max \Big\{ \frac{2(1+C_M)}{\alpha(\omega)}  , \frac{1}{\alpha^2(\omega)} \Big\}.
\end{equation}
As a consequence of  \cite[ Theorem 26.1]{Wloka}, for every fixed $\omega$ there exists a unique solution $\hat{u}(\omega) \in \mathcal{W}(\mathbb{V}, \mathbb{V}^*)$. To determine the constant $C(\omega)$ in (\ref{COmega}), we revisit the proof of the general theorem, in a path-wise sense. The general idea is to approximate the solution $\hat{u}(\omega)$ by a sequence $\hat{u}_m(\omega)$, for every $\omega$. Then show the bound of $\|\hat{u}(\omega)\|_{L^2((0,\tau), \mathbb{V})}$ which implies the weak convergence of $\hat{u}_m(\omega)$ to some $z(\omega)$, which also solves the starting equation and hence, because of the uniqueness of the solution, we have $z(\omega) = \hat{u}(\omega)$, for every $\omega$. In the end, one proves that $\hat{u}_m(\omega)$ strongly converges to $\hat{u}(\omega)$ in $L^2((0,\tau), \mathbb{V})$. 

Observe that we can set $k_0(\omega)$ to be zero, by taking $L_1(\omega,t) \equiv L(\omega,t) + k_0(\omega) Id$. We will continue writing $L(\omega, t)$ and exploit that from the bound (\ref{condII}) and (\ref{kOmega}), we have that 
\begin{equation}\label{Lbound}
\|L(\omega)\| \leq C_1(\omega) + k_0(\omega) =: C_L(\omega).
\end{equation}
Let $\omega \in \Omega$ be arbitrary but fixed, and define 
\[
\hat{u}_m(\omega, t) := \sum_{i=1}^m g_{im}(\omega, t) w_i
\]
where $\{w_i\}_i$ is an ONB of $\mathbb{V}$ and $g_{im} \in L^2(\Omega, \mathbb{V})$ are such that $\hat{u}_m(\omega)$ is a path-wise solution of the ODE
\begin{align*}
\frac{d}{dt} \hat{u}_m(\omega)(\omega,t) + L(\omega, t) \hat{u}_m(\omega,t) &= f(\omega, t) \\
\hat{u}_m(\omega,0) &= u_{0m}
\end{align*}
where $u_{0m} \to u_0$ in $\mathbb{H}$, as $m \to \infty$ and hence $\| u_{0m}\|_{\mathbb{H}} \leq \|u_0\|_{\mathbb{H}} + M$, for some deterministic $M>0$.
Such  $\hat{u}_m(\omega)$ is uniquely determined and it holds
\begin{align*}
\|\hat{u}_m(\omega,\tau)\|_{\mathbb{H}}^2 \, + \, &2 \alpha(\omega) \int_0^\tau \| \hat{u}_m(\omega,t) \|^2_{\mathbb{V}} dt \leq \\
&\|\hat{u}_{0m}\|^2_{\mathbb{H}} + \alpha(\omega) \int_0^\tau \| \hat{u}_m(\omega,t) \|^2_{\mathbb{V}} dt + \frac{1}{\alpha(\omega)} \int_0^\tau \| f(\omega,t) \|_{\mathbb{V}^*}dt.
\end{align*}
Hence, we have
\[
 \int_0^\tau \| \hat{u}_m(\omega,t) \|^2_{\mathbb{V}} dt \leq \max \Big\{ \frac{2(1+C_M)}{\alpha(\omega)}  , \frac{1}{\alpha^2(\omega)} \Big\} \left(\| u(0) \|^2_\mathbb{H}+ \int_0^\tau \| f(\omega,t)\|_{\mathbb{V}^*} dt \right)
\]
where $C_M$ is a deterministic constant such that $M^2 < C_M \|u_0\|^2_{\mathbb{H}}$. 
Since $\hat{u}_m(\omega,t) \to z(\omega)$ in $L^2((0,\tau), \mathbb{V})$, it follows
\begin{equation}\label{zestI}
\int_0^\tau \|z(\omega,t)\|_{\mathbb{V}}^2 \, dt \leq \mathbb{C}(\omega) \left( \|u_0\|^2_{\mathbb{H}} + \int_0^\tau  \| f(\omega,t)\|_{\mathbb{V}^*} dt \right).
\end{equation}
where $ \mathbb{C}(\omega)$ is defined by (\ref{CComega}).  Uniqueness of the solution  implies that$z(\omega)$ equals $\hat{u}(\omega)$, for every $\omega$.
Next step is to bound its time derivative. Utilizing $\frac{d}{dt} \hat{u}(\omega, t) = L(\omega, t) \hat{u}(\omega, t) + f(\omega,t)$ and estimates (\ref{Lbound}) and  (\ref{condII}), we obtain
\begin{equation}\label{zestII}
\int_0^\tau \left\| \frac{d\hat{u}(\omega, t)}{dt} \right\|^2_{\mathbb{V}^*} dt \leq 2\left(C_L(\omega) \mathbb{C}_2(\omega)+1 \right) \int_0^\tau \left(\| f(\omega, \tau) \|^2_{\mathbb{V}^*}dt + \|u_0\|^2_{\mathbb{H}} \right)
\end{equation}
where $k_0(\omega)$ is defined by (\ref{kOmega}). Combining (\ref{zestI}) and (\ref{zestII}), we get the final estimate for every $\omega$
\[
\| \hat{u}(\omega)  \|_{\mathcal{W}(\mathbb{V}, \mathbb{V}^*)} \leq \left(2(C_L(\omega) \mathbb{C}(\omega)+1) + \mathbb{C}(\omega) \right) (\int_0^\tau \| f(\omega, t) \|^2_{\mathbb{V}^*} dt + \|u_0\|^2_{\mathbb{H}})  
\]
which shows (\ref{COmega}).
\end{proof}

One way to show measurability of the map $\omega \mapsto \hat{u}(\omega)$ is to exploit that $\omega \mapsto (f(\omega), u_0)$ is measurable from $L^2(\Omega)$ to $L^2((0,\tau), \mathbb{V}^*) \times \mathbb{H}$. Moreover, from the previous theorem we have that the map $(f(\omega), u_0) \mapsto \hat{u}(\omega) $ is continuous for every $\omega$. Since, continuous function composed with measurable function is measurable w.r.t. Borel $\sigma-$algebra, the measurability of the solution $\hat{u}$ follows.

\begin{remark} Another approach to prove measurability of the solution $\hat{u}$ is to exploit that $\hat{u}$ is a limit of the $\hat{u}_m$ that is a solution of an ODE. One could show that $\omega \mapsto L(\omega, t)$ is measurable and $(u_0, L) \mapsto \hat{u}_m$ is continuous, which implies the measurability of $\hat{u}_m$, and then the limit is as well measurable. 
\end{remark}

Let for now $f$ be deterministic or  uniformly bounded in $\omega$, then from (\ref{RDapriori}) it follows
\[
\mathbb{E}[\|\hat{u}(\omega)\|_{\mathcal{W}(\mathbb{V}, \mathbb{V}^*)}] \leq \mathbb{E}[C(\omega)] (\int_0^\tau \| f(\omega, t) \|^2_{\mathbb{V}^*} dt + \|u_0\|^2_{\mathbb{H}}  )
\] 
Utilizing Cauchy-Schwartz inequality, we determine assumption on the path-wise constants that will ensure that $\mathbb{E}[C(\omega)]$ is finite. Namely, we need $\mathbb{C} \in L^2(\Omega)$ and $C_L(\omega) \in L^2(\Omega)$, where $C_L(\omega)$ is defined by (\ref{Lbound}).

\begin{ass}\label{ass:constMoments}
Assume $\overline{\sigma}^4 \in L^2(\Omega)$ and $C_L\in L^2(\Omega)$.
\end{ass}

\begin{cor}
Assuming in addition to  conditions from Theorem \ref{T2}, that  $f \in L^2((0,\tau), \mathbb{V}^*)$ and Assumption \ref{ass:constMoments}, it holds $\hat{u} \in \mathcal{W}(L^2(\Omega, \mathbb{V}), L^2(\Omega, \mathbb{V}^*))$.
\end{cor}
Note that we can also assume that $f \in L^2((0,\tau), L^2(\Omega,\mathbb{V}^*))$, without assuming uniform bound, and then apply Cauchy-Schwartz inequality. In this case we need more regularity assumptions than in Assumption \ref{ass:constMoments}. Moreover, analogue $p-$moment bounds, $p \in [1, \infty)$ can be obtain, utilizing H\"older's inequality and  modifying the regularity assumption.

\section{Parabolic Stokes equation}
As already announced in the introduction, linear parabolic Stokes equation is one of the examples where it is not enough to consider just plain pull-back transformation of the function. The issue is that  we would like to preserve the divergence free property, and this is not the case if we use the plain pull-back transformation. Instead, motivated by the work presented in \cite{CSZ, WI, Saal} we consider the Piola type transformation of the function. Note that this transformation will keep the divergence free property just in the case of the volume preserving transformation. We consider the following parabolic Stokes equation on a random non-cylindrical domain
\begin{equation}\label{eq:MovingStokesRandom}
\begin{split}
	\partial_t u - \Delta u + \nabla p &= f
	\\
	\text{div} u &= 0
	\\
	u &=0
	\\
	u|_{t=0} &= u_0
\end{split}
\begin{split}
	~&\mbox{in}~ Q_T(\omega),
	\\
	~&\mbox{in}~ Q_T(\omega),
	\\
	~&\mbox{in}~ \partial Q_T(\omega),
	\\
	~&\mbox{in}~ D_0,
\end{split}
\end{equation}
with velocity field $u$ and pressure $p$. Without loss of generality, we consider the system (\ref{eq:MovingStokesRandom}) with
zero boundary conditions. In this setting we make additional assumptions about the transformation. As outlined in \cite[Assum. 1.1]{Saal}, we  assume that the transformation $T(\omega, t, \cdot)$ is a $C^3-$diffeomorphism and it is volume preserving  $\nabla_x T(\omega, t, x) \equiv 1$. We define the transformation $\mathcal{F}$ of the function $u$ by
\[
\hat{u}(\omega, t, \hat{x}) := DT(\omega, t, x) u(\omega, t, T(\omega, t,x)).
\]
From \cite[Prop 2.4]{WI}, it follows that $\text{div} \hat{u} = \text{div} u$ on $Q_T(\omega)$ for any $\omega$. Uilizing  the previous transformation, we obtain the following transformed equation on $Q_0$
	\begin{align*}
	\int_{\Omega}\int_{D_0} \det(D\phi)(
	\hat{u}_t
	+ D\phi^{-1} D\phi_t \hat{u}
	-D\phi^{-1} D\phi^{-T} D(D\phi \hat{u}) \phi_t
	&+ D\phi^{-1}D\phi^{-T} \nabla \hat{p} ) \cdot v\\
	-  D\phi^{-1} \textrm{div} \left( \det(D\phi) D\phi^{-1}D\phi^{-T} D (D\phi \hat{u})\right)\cdot v
	&= \int_{\Omega}\int_{D_0} D\phi^{-1}\hat{f}\cdot v \det(D\phi),
	\end{align*}
for divergence free test functions $v$.
 Observe that in \cite{WI, Saal} authors consider different approach, namely  the Helmholtz projection and Cauchy problem formulation. 
Note that for the volume preserving transformations, the pressure term disappears in the weak form for the divergence free test functions.
The example of these  transformation is rotation about the $z-$axis by a random angle or any arbitrary affine-orthogonal transformation. In this case, the well-posedness of the equation in the random setting can be showed by the standard PDEs techniques. Notice that in \cite{WI}, the local existence is showed for any volume preserving transformation. However, in our setting the difficulty is that we consider this problem path-wise, hence the local time will depend on a sample. For this reason, the general case of a volume preserving transformation is out of the scope of this work and is a topic of a future research.

\section*{Acknowledgments}

The author would like to express her gratitude to Christof Schwab and Carsten Gr\"{a}ser for many helpful suggestions and insightful discussions during the preparation of this paper. Furthermore, I would like to express my gratitude for the  support and help to Nikolas Tapia, concerning the example of not uniformly bounded random evolving field and Philip Harbert, concerning the linear Stokes equation.

\bibliographystyle{plain}
\bibliography{references}

\end{document}